\documentclass[final, a4paper, twoside]{elsarticle}

\usepackage[latin1]{inputenc}
\usepackage{subfigure}
 
\usepackage{amssymb} 
\usepackage{amsmath} 
\usepackage{amsfonts} 
\usepackage{amsthm}
\usepackage{upref} 
\usepackage{graphicx}
\usepackage{color}
\usepackage{url}
\usepackage{paralist} 
\usepackage{etoolbox} 
\usepackage{tikz}
\usepackage{blkarray} 
\usepackage{ifthen}
\usepackage[english]{babel}
\usepackage{calc}
 
 \usepackage[all]{xy}

 \usepackage{mathrsfs} 
\usepackage[all]{xy}
\usepackage{hyperref}

\usepackage{geometry}
\usepackage{fancyhdr}
\hypersetup{pdfauthor={Matthias Lenz},pdftitle={}, 
  pdfkeywords={} }

\newtheorem{Theorem}{Theorem}[section]
\newtheorem{Corollary}[Theorem]{Corollary}
\newtheorem{Proposition}[Theorem]{Proposition}
\newtheorem{Lemma}[Theorem]{Lemma}

\theoremstyle{definition}
\newtheorem{Definition}[Theorem]{Definition}

\newtheorem{Question}[Theorem]{Question}
\newtheorem{Example}[Theorem]{Example}

\theoremstyle{remark}
\newtheorem{Remark}[Theorem]{Remark}

\DeclareMathOperator{\characteristic}{char}
\DeclareMathOperator{\RGr}{RGr} 
\DeclareMathOperator{\Gr}{Gr} 
\DeclareMathOperator{\St}{St} 
\DeclareMathOperator{\GL}{GL}
\DeclareMathOperator{\SL}{SL}

\newcommand{\abs}[1]{\left|#1\right|}
\DeclareMathOperator{\rank}{rk}

\DeclareMathOperator{\diag}{diag} 
\DeclareMathOperator{\lift}{lift} 
\DeclareMathOperator{\trop}{trop} 
 
\newcommand{\WLOG}{WLOG} 
\newcommand{\st}{s.\,t.\ } 
\newcommand{\ie}{\textit{i.\,e.\ }} 
\newcommand{\eg}{\textit{e.\,g.\ }} 

\newcommand{\Z}{\mathbb{Z}}
\newcommand{\Q}{\mathbb{Q}}
\newcommand{\R}{\mathbb{R}}
\newcommand{\CC}{\mathbb{C}}
\newcommand{\F}{\mathbb{F}}
\newcommand{\K}{\mathbb{K}}
\newcommand{\Acal}{\mathcal{A}}

\newcommand{\Mcal}{\mathcal{M}}

\newcommand{\Tcal}{\mathcal{T}}
\newcommand{\Rcal}{\mathcal{R}}

\DeclareMathOperator{\GPr}{{{GP\!_r}}}
\DeclareMathOperator{\GP}{{{GP}}}
\DeclareMathOperator{\sgn}{{{sgn}}}

\DeclareMathOperator{\spa}{{{span}}}

\newcommand{\Gcal}{\mathcal G}

\newcommand{\latproj}[1]{{\bar #1}} 

\DeclareMathOperator{\sym}{Sym} 
  
\tikzstyle{bonn1colour}=[red!90!black]
\tikzstyle{bonn2colour}=[cyan!78!black]
\tikzstyle{bonn3colour}=[green!80!black]
\tikzstyle{bonn4colour}=[blue!90!black]

\definecolor{darkgreen}{rgb}{0,0.8,0}

\date{\today}

\begin{document}
\newgeometry{twoside}  
\title
{%
  On powers of Pl\"ucker coordinates and representability of arithmetic matroids\tnoteref{t1}
}
\author{Matthias Lenz\fnref{fn1}}
\address{Universit\'e de Fribourg, D\'epartement de Math\'ematiques, 1700 Fribourg, 
Switzerland}
\ead{maths@matthiaslenz.eu}

\fntext[fn1]{%
The author
was supported by a
fellowship within the postdoc programme of the German Academic Exchange Service (DAAD).%
}
\tnotetext[t1]{%
\textcopyright 2019. This manuscript version is made available under the CC-BY-NC-ND 4.0 license \url{http://creativecommons.org/licenses/by-nc-nd/4.0/}.
}

\begin{keyword}
Grassmannian \sep Pl\"ucker coordinates \sep arithmetic matroid \sep regular matroid \sep representability \sep vector configuration
 
\MSC[2010]{Primary: 
05B35 \sep 
14M15.  
Secondary:
14T05. 
}
\end{keyword}

\begin{abstract}
The first problem we investigate is the following: given $k\in \R_{\ge 0}$ and a vector $v$ of 
 Pl\"ucker coordinates of a point in the real Grassmannian, is the vector obtained by taking the $k$th power of each entry of $v$ again a vector of Pl\"ucker coordinates?
 For $k\neq 1$, this is true if and only if the  corresponding matroid is regular. Similar results hold over other fields.  We also describe the subvariety of the Grassmannian that consists of all the points that define a regular matroid.

 The second topic is a related problem for arithmetic matroids. Let $\Acal = (E, \rank, m)$ be an arithmetic matroid and let $k\neq 1 $ be a non-negative integer.
 We prove that if $\Acal$ is representable and the underlying matroid is non-regular, then $\Acal^k := (E, \rank, m^k)$ is not representable.
 This provides a large class of examples of arithmetic matroids that are not representable. On the other hand, if the underlying matroid is regular and an additional condition is satisfied, then $\Acal^k$ is representable.
 Bajo--Burdick--Chmutov have recently discovered that arithmetic matroids of type $\Acal^2$ arise naturally in the study of 
 colourings and flows on CW complexes. 
 In the last section, we prove a family of necessary conditions for representability of arithmetic matroids.
\end{abstract}

\maketitle
\thispagestyle{plain}
\pagestyle{fancy}

\section{Introduction}

 \begin{Theorem}
 \label{Theorem:NonRegularMatricesNew}
  Let $X$ be a $(d\times N$)-matrix of full rank $d \le N$ with entries in $\R$.
  Let $k\neq 1$ be a non-negative real number.
  Then $X$ represents a regular matroid if and only if the following condition is satisfied:
  there is a $(d\times N$)-matrix $X_k$ with entries in $\R$ \st
  for each maximal minor $\Delta_I(X)$, indexed by $I\in \binom{[N]}{d}$, $\abs{\Delta_I(X)}^k = \abs{\Delta_I(X_k)}$
  holds.
 If $k$ is a non-negative integer, then the same statement holds over any ordered field $\K$.
 \end{Theorem}

 Recall that a ($d\times N$)-matrix $X$ that has full rank $d$ 
 represents a regular matroid if and only if there is 
 a
 totally unimodular ($d\times N$)-matrix $A$ that represents the same matroid, \ie a maximal minor of $X$ is $0$ if and only if 
 the corresponding minor of $A$ is $0$.
 A matrix is totally unimodular if and only if all its non-singular square submatrices have
 determinant $\pm 1$.

  Theorem~\ref{Theorem:NonRegularMatricesNew} can be restated 
  as a result 
  on Grassmannians, which are
   fundamental objects in algebraic geometry (\eg \cite{OrientedMatroidsBook,harris-1995}).
  If $d\le N$, the maximal minors of a $(d\times N)$-matrix $X$ with entries in a field $\K$ are known as the  \emph{Pl\"ucker coordinates}
  of the space spanned by the rows of $X$.
  An element of the  \emph{Grassmannian} $\Gr(d,N)$, \ie the set of all $d$-dimensional subspaces of $\K^N$, 
  is uniquely determined by its Pl\"ucker coordinates (up to a scalar, non-zero multiple).
  This yields an embedding of the Grassmannian into $(\binom{N}{d}-1)$-dimensional projective space.  
  An element $\xi \in \Lambda^d\K^N$, the $d$th exterior power of $\K^N$, is called an antisymmetric tensor.
  It is decomposable if and only if there are $v_1,\ldots, v_d\in \K^N$ \st
  $\xi = v_1\wedge\ldots\wedge v_d$.
  It is known that 
   \begin{equation}
    \xi = v_1 \wedge \ldots \wedge v_d = \sum_{\substack{I=\{i_1,\ldots, i_d\} \\1\le i_1 < \ldots < i_d \le N}} 
    \Delta_I(X) \: e_{i_1} \wedge \ldots \wedge e_{i_d} \in \Lambda^d \K^N,
 \end{equation}
  where $X$ denotes the matrix whose rows are $v_1,\ldots, v_d$.
  Hence Theorem~\ref{Theorem:NonRegularMatricesNew}  can be restated as a result 
  describing when the $k$th power 
  of a point in the Grassmannian (defined as the element-wise $k$th power of the Pl\"ucker coordinates) or of a decomposable antisymmetric tensor is again a point in the Grassmannian
  or a decomposable antisymmetric tensor, respectively.
  Such a result could potentially be interesting in the context of tropical geometry. 
  Over a suitable field, the operation 
  of taking the $k$th power corresponds  in the tropical setting to scaling  by the factor $k$.
  
  Very recently, Dey--G\"orlach--Kaihnsa studied a more general problem 
  using algebraic methods: they considered coordinate-wise powers of subvarieties of $\mathbb{P}^n$  \cite{dey-goerlach-kaihnsa-2018}.

 \smallskip
 Our second topic
 is the question whether the $k$th power of a representable arithmetic matroid is again representable. 
 We will see that this is in a certain sense a discrete analogue of the first topic.
 An arithmetic matroid $\Acal$ is a triple $(E, \rank, m)$, where $(E,\rank)$ is a matroid on the ground set $E$ with rank function $\rank$ 
 and $m: 2^E\to \Z_{\ge 1}$ is the so-called multiplicity function, that satisfies certain axioms. In the representable case, \ie when the 
 arithmetic matroid is determined by a list of integer vectors,
 this multiplicity function records data such as the 
 absolute value of the determinant of a basis.

 Arithmetic matroids where recently introduced by D'Adderio--Moci 
\cite{moci-adderio-2013}.
 A further generalization are  matroids over a ring by Fink--Moci \cite{fink-moci-2016}.
 Arithmetic matroids capture many combinatorial and topological properties of toric arrangements \cite{callegaro-delucchi-2017,lawrence-2011,moci-tutte-2012}
 in a similar way as matroids carry information about the corresponding hyperplane arrangements \cite{orlik-terao-1992,stanley-2007}.
 Like matroids, they have an important polynomial invariant called the Tutte polynomial.
 Arithmetic Tutte polynomials also appear in many other contexts, \eg in the study    
 of cell complexes, the 
 theory of vector partition functions, and Ehrhart theory of zonotopes 
 \cite{bajo-burdick-chmutov-2014,lenz-arithmetic-2016,stanley-1991}.

 Let $\Acal=(M, \rank, m)$ be an arithmetic matroid and let $k \in \Z_{\ge 0}$.
 Delucchi--Moci 
 have shown that $\Acal^k :=(M, \rank, m^k)$
 is also an arithmetic matroid \cite{delucchi-moci-2018}.
 We address the following question in this paper:
 given a representable arithmetic matroid $\Acal$, is $\Acal^k$ also representable?
 If the underlying matroid is non-regular, this is false. On the other hand, if the underlying matroid is regular 
 and an additional condition (weak multiplicativity) is satisfied, it is representable and we are able to calculate a representation.
  This leads us to define the classes of weakly and strongly multiplicative arithmetic matroids.
  It turns out that arithmetic matroids that are both regular and weakly/strongly multiplicative play a role
  in the theory of arithmetic matroids
  that is
  similar to the role of regular matroids in matroid theory: they arise from totally unimodular matrices
  and they preserve some nice properties of  arithmetic matroids defined by a labelled graph  \cite{moci-adderio-flow-2013},
 just like regular matroids generalize and preserve nice properties of  graphs/graphic matroids. 
 Furthermore, representations of weakly multiplicative arithmetic matroids are unique, up to some obvious transformations \cite{lenz-unique-2019}.

 Recently, various authors have extended notions from graph theory such as 
 spanning trees, colourings, and flows 
 to higher dimensional cell complexes (\eg \cite{beck-breuer-godkin-martin-2014,duval-klivans-martin-2011,duval-klivans-martin-2015}),
 where graphs are being considered as $1$-dimensional cell complexes.
 Bajo--Burdick--Chmutov introduced the 
   modified $j$th Tutte--Krushkal--Renardy (TKR) polynomial of a cell complex, which 
   captures the number of cellular $j$-spanning
trees, counted with multiplicity the square of the cardinality of the torsion part of a certain homology group \cite{bajo-burdick-chmutov-2014}. 
 This invariant was introduced by Kalai \cite{kalai-1983}.
   The   TKR polynomial is the Tutte polynomial of the arithmetic matroid obtained by squaring
   the multiplicity function of the arithmetic matroid defined by the $j$th boundary matrix of the cell complex 
      \cite[Remark~3.3]{bajo-burdick-chmutov-2014}.
      This explains why it is interesting to consider the arithmetic matroid $\Acal^2$, and more generally, the $k$th power of 
      an   arithmetic matroid.

 While representable arithmetic matroids simply arise from a list of integer vectors,
 it is a more difficult task to construct large classes of arithmetic matroids that are not representable.
  Delucchi--Riedel constructed a class of non-representable arithmetic matroids
  that arise from group actions on semimatroids, \eg from arrangements of pseudolines
 on the surface of a two-dimensional torus \cite{delucchi-riedel-2018}.
 The arithmetic matroids of type $\Acal^k$, where $\Acal$ is representable and has an underlying matroid 
 that is non-regular is  a further large class of examples.

  \smallskip
   Let us describe how the objects studied in this paper are related to each other and why 
   the arithmetic matroid setting can be considered to be a discretization of the Pl\"ucker vector setting.
   For a principal ideal domain $R$, let $\St_R(d,N)$ denote the set of all $(d\times N)$-matrices with entries in $R$ that have full rank
   and let $\GL(d,R)$ denote the set of invertible $(d\times d)$-matrices over $R$.
   Let $\Gr_R(d,N)$ denote the Grassmannian over $R$, \ie  $\St_R(d,N)$ modulo a left action of $\GL(d,R)$.
 If $R$ is a field $\K$, this is the usual Grassmannian. 
 Let $\Mcal$ be a matroid of rank $d$ on $N$ elements that is realisable over $\K$.
 The realisation space of $\Mcal$, the set $\{ X \in \St_\K(d,N) : X\text{ represents } \Mcal\}$,
 can have a very complicated structure by an unoriented version of Mn\"ev's universality theorem \cite{OrientedMatroidsBook,mnev-1988}.
 It is invariant under a  $\GL(d,\K)$ action from the left and
 a right action of non-singular diagonal $(N\times N)$-matrices, \ie of the algebraic torus $(\K^*)^N$.
 This leads to a stratification of the  Grassmannian $\Gr_\K(d,N)$ 
 into the matroid strata  $\Rcal(\Mcal) = \{ \bar X \in \K^{d\times N} / \GL(d,\K) : X \text{ represents } \Mcal\}$ \cite{gelfand-gorseky-macpherson-serganova-1987}.
 Arithmetic matroids 
 correspond to the setting $R=\Z$:
 the set of representations of a fixed  torsion-free arithmetic matroid $\Acal$ of rank $d$ on $N$ elements
 is a subset of $ \St_{\Z}(d,N) $ that is invariant under a left action of
 $ \GL(d, \Z) $  and a right action of diagonal $(N\times N)$-matrices with entries in $\{\pm 1\}$, \ie of
 $(\Z^*)^N$, the maximal multiplicative subgroup of $\Z^N$.
  This leads to a stratification of the discrete Grassmannian $\Gr_\Z(d,N)$ into arithmetic matroid strata $\Rcal(\Acal)=
  \{ \bar X \in \St_\Z(d,N) / \GL(d,\Z) : X \text{ represents } \Acal \}$.

\subsection*{Organisation of this article}
 The remainder of this article is organised as follows.
 In Section~\ref{Section:Results} we will state and discuss 
 our main results on representability of 
 powers of Pl\"ucker coordinates and powers of arithmetic matroids and we will
 give some examples.
 In Section~\ref{Section:SubvarietiesGrassmannian} we will define and study    two related subvarieties of the 
 Grassmannian: the regular Grassmannian, \ie the set of all points that correspond to a regular matroid
 and the set of Pl\"ucker vectors that have a $k$th power.
 The mathematical background will be explained in Section~\ref{Section:Background}.
 The main results will be proven in Sections~\ref{Section:Proofs} to~\ref{Section:MultiplicativeRegularAriMatroids}.
 In Section~\ref{Section:Towards} we will prove some necessary conditions for the representability of arithmetic matroids
 that are derived from the Grassmann--Pl\"ucker relations.

\section{Main results}
\label{Section:Results}

In this section we will present our main results.
 We will start with the results on powers of Pl\"ucker coordinates in Subsection~\ref{Subsection:Matrices}.
 We will present the analogous results on arithmetic matroids in Subsection~\ref{Subsection:ArithmeticMatroids}.
  See Figure~\ref{Figure:Overview} for an overview of our results in both settings.
  In Subsection~\ref{Subsection:TwoDifferentMultiplicityFunctions} we will examine which of the results on arithmetic matroids
  still hold if there are two different multiplicity functions and in Subsection~\ref{Subsection:LabelledGraphsResults} we will
  study the special case of arithmetic matroids defined by a labelled graph.
 In Subsection~\ref{Subsection:Examples} we will give some additional examples.
 
 \smallskip
 Recall that a matroid is \emph{regular} if it can be represented over every field,
 or equivalently, if it can be represented by a totally unimodular matrix.
 It will be important that
 a regular matroid can also be represented by a matrix that is not totally unimodular.
 This is true for all examples in Subsection~\ref{Subsection:Examples}.
 Two other characterisations of regular matroids are given in Subsection~\ref{Subsection:RegularMatroids}.

\subsection{Powers of Pl\"ucker vectors}
 \label{Subsection:Matrices}
  For a ring $R$ (usually $R=\Z$ or $R$ a field) and integers $d$ and $N$, we will write $R^{d\times N}$ to denote the set of $(d\times N)$-matrices with entries in $R$.
  If $E$ is a set of cardinality $N$, each $X\in R^{d\times N}$ 
  corresponds to a list (or finite sequence) $X = (x_e)_{e\in E}$ of $N$ vectors in $R^d$.
  We will use the notions of list of vectors and matrix interchangeably.
  Slightly abusing notation, we will write $X\subseteq R^d$
  to say that $X$ is a list of elements of $R^d$.
  Let $d\le N$ and $I\subseteq \binom{[N]}{d}$, \ie $I$ is a subset of $[N]:=\{1,\ldots, N\}$ of cardinality $d$.
  Then for $X\in R^{d\times N}$, $\Delta_I(X)$ denotes the maximal minor of $X$ that is indexed by $I$, \ie the determinant
  of the square submatrix of $X$ that consists of the columns that are indexed by $I$.

 \begin{Theorem}[non-regular, Pl\"ucker]
 \label{Theorem:NonRegularMatricesNewNew}
  Let $X\in \R^{d\times N}$ be a matrix of full rank $d\le N$.
  Suppose that the matroid represented by $X$ is non-regular.
  Let $k\neq 1$ be a  non-negative real number.
  Then 
  there is no 
  $X_k\in \R^{d\times N}$
  \st
  $\abs{\Delta_I(X)}^k = \abs{ \Delta_I(X_k) }$ holds for each $I\in \binom{[N]}{d}$.  Here, we are using the convention $0^0=0$.
  
  If $k$ is an integer, then the same statement holds over any ordered field $\K$.
 \end{Theorem}

 Throughout this article, we are using the convention $0^0 = 0$. 
 The reason for this is that for any $k$, including the case $k=0$,
 we want that $ \Delta_I(X_k) = \Delta_I(X)^k = 0 $ if $\Delta_I(X)=0$, \ie 
 a non-basis of $X$ should  also be a non-basis of $X_k$.

 Recall that an \emph{ordered field} is a field $\K$ together with a total order $\le$ on $\K$ \st
 for all $a,b,c\in \K$, $a\le b$ implies $a+c \le b+c$ and $0\le a,b$ implies $0\le ab$. 
 For $x\in \K$, the absolute value is defined by
 \begin{equation}
  \abs{x} := \begin{cases}
                x & x \ge 0 \\
                -x & x<0
             \end{cases}.
 \end{equation}

 Now we will see that for a matrix $X$ that represents a regular matroid, there is a matrix
 $X_k$ whose  Pl\"ucker vector is up to sign equal to the $k$th power of the Pl\"ucker vector
 of $X$. For odd exponents $k$ and over ordered fields, we can even fix the signs.
 
\begin{Theorem}[regular, Pl\"ucker]
 \label{Proposition:RegularMatrices}
   Let $\K$ be a field.
   Let $X\in \K^{d\times N}$ be a matrix  of rank $d\le N$.
   Suppose that the underlying matroid is regular.
 
 Then,
  \begin{enumerate}[(i)]
    \item for any integer $k$, there is $X_k \in \K^{d\times N}$ 
    \st 
      the Pl\"ucker coordinates satisfy 
      $\Delta_I(X)^k = \pm \Delta_I(X_k)$  for any $I\in \binom{[N]}{d}$.
      \label{item:RegPlueckerNormal}
    \item 
    If $k$ is odd, then there is $X_k \in \K^{d\times N}$  \st the following stronger equalities hold for all $I\in \binom{[N]}{d}$: $\Delta_I(X)^k = \Delta_I(X_k)$.
  \item 
  \label{enumerate:RegPlueckerOrderedField}
  If  $\K$ is an ordered field,
  then for all $k$ \st the expression $x^k$ is well-defined for all positive elements of $\K$ (\eg $k\in \Z$ or  $\K=\R$ and $k\in \R$), 
  there is   $X_k\in \K^{d\times N}$ 
    \st 
      $\abs{\Delta_I(X)}^k = \abs{\Delta_I(X_k)}$  and  $ \Delta_I(X_k) \cdot \Delta_I(X) \ge 0 $
      for all  $I\in \binom{[N]}{d}$.
  \end{enumerate}
       Here, we are  using the convention $0^0=0$.
\end{Theorem}

 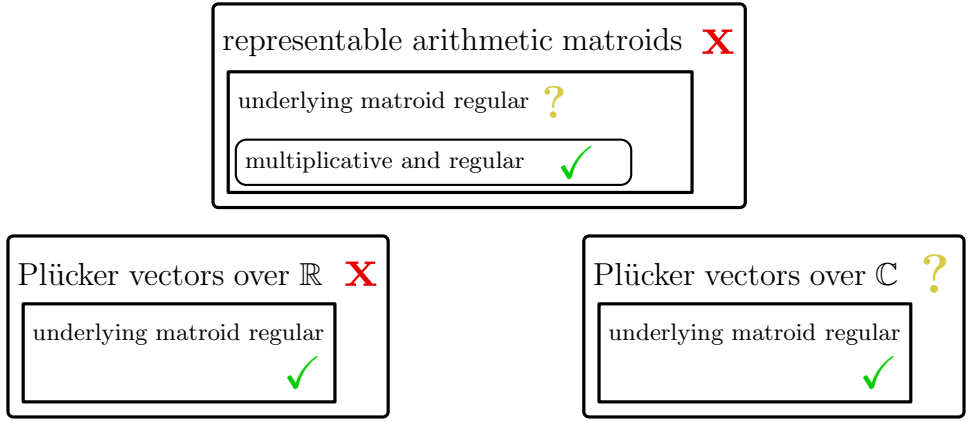
\begin{figure}[t]
 \newcommand{\all}{\Large{representable arithmetic matroids}}
 \newcommand{\sometimes}{underlying matroid regular}
 \newcommand{\yes}{multiplicative and regular}
 \newcommand{\breitebox}{8.5} %
 \newcommand{\textboxwidth}{7.5cm} 
 \newcommand{\reduceheight}{0} %
 \newboolean{regularyes} %
 \setboolean{regularyes}{false}
 \newboolean{alwaysno} %
 \setboolean{alwaysno}{true}
 \begin{center}
  \begin{tikzpicture}[line join=round, scale=1,
 vertex/.style={color=black}]
  %
%
%
 \draw[very thick, rounded corners=2pt] (0,0) rectangle (\breitebox,2.7 - \reduceheight); 
 \node[anchor = west, text width=\textboxwidth] (a) at (0, 2.2 - \reduceheight) {\all}; %
  %
%
 \draw[very thick] (0.2,0.2) rectangle (\breitebox - 0.7, 1.8 - \reduceheight);
 \node[anchor = west] (b) at (0.2, 1.4 - \reduceheight) {\small \sometimes}; 
 \ifthenelse{\boolean{regularyes}}
 {
    \node[anchor = west, green!80!black] (a) at (\breitebox - 1.5, 0.6) {\huge\checkmark};
 }
 {
    \node[anchor = west, yellow!80!black] (a) at (4.2, 1.4) {\huge \bfseries ?};
    \draw[thick, rounded corners=4pt] (0.3,0.3) rectangle (5.5,0.9); %
    \node[anchor = west] (b) at (0.3, 0.6) {\small \yes};
    \node[anchor = west, green!80!black] (a) at (4.4, 0.6) {\huge\checkmark}; %
 }
  \ifthenelse{\boolean{alwaysno}}
  {
   \node[anchor = west, red!90!black] (a) at (\breitebox - 0.7, 2.2 - \reduceheight) {\Huge\bfseries  x};
  }
  {
   \node[anchor = west, yellow!80!black] (a) at (\breitebox - 0.7, 2.2 - \reduceheight) {\Huge\bfseries  ?};
  }
 \end{tikzpicture}
\\[3mm]
\renewcommand{\textboxwidth}{6.5cm}
 
 \renewcommand{\breitebox}{6}
 \renewcommand{\reduceheight}{0.3} %
 \renewcommand{\all}{\Large Pl\"ucker vectors over $\R$}
 \setboolean{regularyes}{true}
  \begin{tikzpicture}[line join=round, scale=1,
 vertex/.style={color=black}]
  %
%
%
 \draw[very thick, rounded corners=2pt] (0,0) rectangle (\breitebox,2.7 - \reduceheight); 
 \node[anchor = west, text width=\textboxwidth] (a) at (0, 2.2 - \reduceheight) {\all}; %
  %
%
 \draw[very thick] (0.2,0.2) rectangle (\breitebox - 0.7, 1.8 - \reduceheight);
 \node[anchor = west] (b) at (0.2, 1.4 - \reduceheight) {\small \sometimes}; 
 \ifthenelse{\boolean{regularyes}}
 {
    \node[anchor = west, green!80!black] (a) at (\breitebox - 1.5, 0.6) {\huge\checkmark};
 }
 {
    \node[anchor = west, yellow!80!black] (a) at (4.2, 1.4) {\huge \bfseries ?};
    \draw[thick, rounded corners=4pt] (0.3,0.3) rectangle (5.5,0.9); %
    \node[anchor = west] (b) at (0.3, 0.6) {\small \yes};
    \node[anchor = west, green!80!black] (a) at (4.4, 0.6) {\huge\checkmark}; %
 }
  \ifthenelse{\boolean{alwaysno}}
  {
   \node[anchor = west, red!90!black] (a) at (\breitebox - 0.7, 2.2 - \reduceheight) {\Huge\bfseries  x};
  }
  {
   \node[anchor = west, yellow!80!black] (a) at (\breitebox - 0.7, 2.2 - \reduceheight) {\Huge\bfseries  ?};
  }
 \end{tikzpicture}
\\[3mm]
\setboolean{alwaysno}{false}
 \renewcommand{\all}{\Large Pl\"ucker vectors over $\CC$}
  \begin{tikzpicture}[line join=round, scale=1,
 vertex/.style={color=black}]
  %
%
%
 \draw[very thick, rounded corners=2pt] (0,0) rectangle (\breitebox,2.7 - \reduceheight); 
 \node[anchor = west, text width=\textboxwidth] (a) at (0, 2.2 - \reduceheight) {\all}; %
  %
%
 \draw[very thick] (0.2,0.2) rectangle (\breitebox - 0.7, 1.8 - \reduceheight);
 \node[anchor = west] (b) at (0.2, 1.4 - \reduceheight) {\small \sometimes}; 
 \ifthenelse{\boolean{regularyes}}
 {
    \node[anchor = west, green!80!black] (a) at (\breitebox - 1.5, 0.6) {\huge\checkmark};
 }
 {
    \node[anchor = west, yellow!80!black] (a) at (4.2, 1.4) {\huge \bfseries ?};
    \draw[thick, rounded corners=4pt] (0.3,0.3) rectangle (5.5,0.9); %
    \node[anchor = west] (b) at (0.3, 0.6) {\small \yes};
    \node[anchor = west, green!80!black] (a) at (4.4, 0.6) {\huge\checkmark}; %
 }
  \ifthenelse{\boolean{alwaysno}}
  {
   \node[anchor = west, red!90!black] (a) at (\breitebox - 0.7, 2.2 - \reduceheight) {\Huge\bfseries  x};
  }
  {
   \node[anchor = west, yellow!80!black] (a) at (\breitebox - 0.7, 2.2 - \reduceheight) {\Huge\bfseries  ?};
  }
 \end{tikzpicture}
%
\end{center}
 \caption{A summary of our results about when taking a $k$th power preserves being a representable arithmetic matroid / 
 a Pl\"ucker vector.
 \checkmark\  means yes, {\bfseries ?} means sometimes, and {\bfseries x} means never.}
 \label{Figure:Overview}
\end{figure}

\begin{Example}
\label{Example:SquaringOneDimCase}
  Let us consider the case $d=1$, \ie $ X = ( a_1, \ldots, a_N )$ with $a_i\in \K$.
  Then $ X_k = ( a_1^k, \ldots, a_N^k )$.
\end{Example}
 The method of taking the $k$th power of each entry does of course not generalize to higher dimensions. 
 However, for matrices that represent a regular matroid, a slightly more refined method works:
 below, we will see that one can write $ X = T A D $, with $ T\in \GL(d,\K)$, $A$ totally unimodular, and $ D \in \GL(N,\K) $ a diagonal matrix. 
 Then the matrix $X_k := T^k A D^k$ can be used.

\begin{Example}
 In Theorem~\ref{Proposition:RegularMatrices}\eqref{item:RegPlueckerNormal}, we really need  
 that the equalities hold ``up to sign''. Otherwise,
 the result is wrong.
 Let $k=2$ and let
 \begin{equation}
   X = \begin{pmatrix}
         1 & 0 & 1 & 0 \\
         0 & 1 & 1 & 1
       \end{pmatrix}.  
 \end{equation}
All maximal minors are equal to $1$, except for
$\Delta_{23}=-1$ and $\Delta_{24}=0$. So the Grassmann--Pl\"ucker relation can be simplified to 
\begin{equation}
\label{equation:GPexample}
 \Delta_{12}\Delta_{34} + \Delta_{14}\Delta_{23}=0. 
\end{equation}
 If we square the Pl\"ucker coordinates, they all become positive. But then
  \eqref{equation:GPexample} is no longer satisfied.
\end{Example}

  Theorem~\ref{Theorem:NonRegularMatricesNewNew} is in many ways best possible.
  For example, 
  over $\F_2$, for any integer $k\ge 0$, every Pl\"ucker  vector has a $k$th power as minors can only be  $0$ or $1$, so $X=X_k$ always holds.
    Similarly, over $\F_3$, minors can only be $0$,$1$, or $-1$, so every Pl\"ucker vector has a $k$th power up to sign for any integer $k\ge 0$. 
  We are also able to construct counterexamples over many  other fields.

 \begin{Proposition}
 \label{Proposition:NonRegularMatricesNewNewBestPossible}
     Theorem~\ref{Theorem:NonRegularMatricesNewNew} is in many ways best possible.
     Namely, it does not hold in positive characteristic or
     over algebraically closed fields. It is also wrong over $\R$ for negative real numbers $k$.

     More specifically, we are able to  construct counterexamples for any field of characteristic $p > 0$;
     over $\CC$ for any $k \in \Z_{\ge 2}$, and more generally,
     for algebraically closed fields of characteristic $p\ge 0$ and any $k \in \Z_{\ge 3}$, 
     if $p\neq 2$ and $p$ does not divide $k$; over $\R$ and any negative real number $k$. 
 \end{Proposition}

 \begin{Remark}
   We have seen that given a \emph{real} Pl\"ucker vector $\xi = (\xi_I: I\in \binom{[N]}{d})$,  
    whether $\xi^k:=(\xi_I^k : I\in \binom{[N]}{d})$ is also contained in the Grassmannian depends only on the matroid defined by $\xi$.

    Over other fields, this is no longer the case.
    For example, in the proof of  Proposition~\ref{Proposition:NonRegularMatricesNewNewBestPossible} we will see 
    that there is a point $\xi$ in the $U_{2,4}$ stratum of  $\CC$ \st $\xi^2$ is again a Pl\"ucker vector.
    However, there are other points in the same stratum whose square is not a Pl\"ucker vector. 
    For example, using the notation of  \eqref{equation:CounterexampleMatrices}, the point corresponding to the matrix $X(-1)$.
 \end{Remark}

\begin{Remark}
  The operation of taking powers of  Pl\"ucker vectors corresponds to scaling  in tropical geometry \cite{maclagan-sturmfels-2015}, when working over a suitable field.
  For example, let $\K$  be the field of Puiseux series over some field $K$ that is equipped with a valuation that is trivial 
  on $K$.
  As usual, we choose a matrix $ X \in \K^{d\times N} $ and we consider the Pl\"ucker vector $ \Delta(X) \in \K^{\binom{[N]}{d}} $.
  After tropicalization, this corresponds to a 
  tropical linear space, or equivalently, a 
  regular subdivision of the hypersimplex into matroid polytopes (\eg \cite{speyer-2005}).    
  On the tropical side, the operation of taking the $k$th power corresponds to scaling 
  the tropical Pl\"ucker vector by the factor $k$.
  However, the situation in the tropical setting is different: there is always $X_k \in \K^{d\times N}$ \st 
  $\trop(\Delta(X_k)) = k \cdot \trop(\Delta(X))$.
\end{Remark}

 \subsection{Powers of arithmetic matroids}
 \label{Subsection:ArithmeticMatroids}
 
 In this subsection we will present results similar to the ones in the previous subsection in the setting of representable arithmetic matroids.
 Let $\Acal = (E, \rank, m)$ be an arithmetic matroid.
 For an integer $k\ge 0$, we will consider the arithmetic matroid
 $\Acal^k:= (E, \rank, m^k)$, where $m^k(e):= (m(e))^k$ for  $e\in E$.
 This is indeed an arithmetic matroid 
 (\cite[Corollary~5]{backman-lenz-2016}, \cite[Theorem~2]{delucchi-moci-2018}).
\begin{Theorem}[non-regular, arithmetic matroid]
\label{Theorem:NonRegularNonRepresentable}
 Let $\Acal = (E, \rank, m)$ be an arithmetic matroid.
 Suppose that $\Acal$ is representable and the matroid $(E, \rank)$ is non-regular.
 Let $k \neq 1$ be a non-negative integer.
 Then $\Acal^k:= (E, \rank, m^k)$ is not representable.
\end{Theorem}

 Recall that an arithmetic matroid is \emph{torsion-free} if
 $m(\emptyset)=1$.
 In the representable case, this means that the arithmetic matroid can be represented by a list of vectors in a lattice, \ie a 
  finitely generated abelian group that is torsion-free.
 The torsion-free case is often simpler and a reader with little knowledge about arithmetic matroids is invited to consider only this case.

 Let $\Acal=(E,\rank, m)$ be a torsion-free arithmetic matroid.
 Let $I\subseteq E$ be an independent set. We say that $I$  is \emph{multiplicative} 
 if it  satisfies $m(I) =  \prod_{x\in I} m(\{ x \})$.
 This condition is always satisfied if $m(I)=1$.

\begin{Definition}
     We call a torsion-free arithmetic matroid \emph{weakly multiplicative}
     if it has at least one multiplicative basis.
     In general, we call an arithmetic matroid $\Acal$
     weakly multiplicative if there is a torsion-free arithmetic matroid $\Acal'$ 
     with a multiplicative basis \st $\Acal= \Acal'/Y$, where  $Y$ denotes a subset of the multiplicative basis. 

     We call a torsion-free arithmetic matroid \emph{strongly multiplicative} %
     if all its bases  are multiplicative.
     In general, we call an arithmetic matroid strongly multiplicative if it is  a quotient of a torsion-free arithmetic matroid whose bases are all multiplicative.
 \end{Definition}

 Let $\Acal= (E,\rank,m)$ be an arithmetic matroid.
 Let $Y$ be a set \st $ Y \cap E = \emptyset $. We call $\Acal'=(E\cup Y, \rank', m')$ a \emph{lifting} of $\Acal$ if 
 $\Acal' / Y = \Acal$.  \label{page:FirstDefinitionOfLifting}
 Similarly, if  $ X\subseteq \Z^d \oplus \Z_{q_1} \oplus \ldots \oplus \Z_{q_n} $ represents $\Acal$ 
 and $\lift(X)\subseteq \Z^{d+n}$ represents $\Acal'$, we call $\lift(X)$ a lifting   of $X$.
 If $\Acal'$ is a lifting of $\Acal$, then $\Acal$ is called a \emph{quotient} of $\Acal'$.
 The contraction operation $/$ is explained in Subsection~\ref{Subsection:RestrictionContraction}.

 Let $\Acal$ be an arithmetic matroid. If $\Acal$ is torsion-free, we call it regular  if it is representable and its underlying matroid is regular.
 In general, we call an arithmetic matroid  \emph{regular} if it has a lifting that is a torsion-free and regular arithmetic matroid.
 Note that this condition is slightly stronger than having an underlying matroid that is regular\footnote{For an example, consider the arithmetic matroid
 defined by the list $X=((1,\bar 0),(0,\bar 1), (1, \bar 1), (-1, \bar 1))\subseteq \Z \oplus \Z_2$.
 The underlying matroid is $U_{1,4}$, which is regular, but all of its liftings contain a $U_{2,4}$.}. 
 Below, we will consider arithmetic matroids that are \emph{regular and weakly/strongly multiplicative}.
 If there is torsion, we assume that for such arithmetic matroids, there is \emph{one} torsion-free lifting that is \emph{both} 
 regular and satisfies the multiplicativity condition.

 We will see that regular and strongly/weakly multiplicative arithmetic matroids have representations with special properties
 that will allow us to prove the following theorem.

\begin{Theorem}[regular, arithmetic matroid]
 \label{Theorem:PowerOfArithmeticMatroid}
  Let $\Acal$ be an arithmetic matroid that is regular  and weakly multiplicative.
  Let $k\ge 0$ be an integer. 
  Then $\Acal^k$ is representable.
\end{Theorem}
  The proof of Theorem~\ref{Theorem:PowerOfArithmeticMatroid} is constructive:
  we will see below that we can write $X$ as a quotient of a doubly scaled unimodular list, \ie we can write
  $X= (T \cdot A \cdot D) / Y$, where $T$, $A$, $D$, are matrices with special properties and $Y$ is a special sublist of $TAD$. 
  We will see that $\Acal^k$ is represented by the list $X_k:= (T^k \cdot A \cdot D^k) / Y_k$, where $Y_k$ is obtained from $Y$ by multiplying
  each entry by $k$.

\smallskip
 We say that a list $X\subseteq \Z^d$ is 
 a \emph{scaled unimodular list} 
 if  there is a matrix $ A \in \Z^{d\times N} $ that is totally unimodular and  a non-singular diagonal matrix $ D \in \Z^{N\times N} $  
 \st  $X = A D$.
 A list $ X \subseteq \Z^{d} \oplus \Z_{q_1} \oplus \ldots \Z_{q_n} $  is a \emph{quotient of a scaled unimodular list}
 (QSUL)
 if there is a scaled unimodular list $\lift(X) \in \Z^{ (d + n) \times (N + n)}$ and a sublist $Y\subseteq \lift(X)$ \st
 $ X = \lift(X) / Y$.
 The quotient $\lift(X) / Y$ denotes the image of $ \lift(X) \setminus Y $ in $ \Z^{d+n} / \langle Y \rangle$ under the canonical projection. 
 $\langle Y\rangle\subseteq \Z^{d+n}$ denotes the subgroup generated by $Y$.

  We say that a list $X\subseteq \Z^{d}$ is a \emph{doubly scaled unimodular list} 
 if there is a matrix $A \in \Z^{d\times N}$ that is totally unimodular and (after reordering its columns) its first $d$ columns form an identity matrix and
 non-singular diagonal matrices $D\in \Q^{N\times N}$, $T\in \Q^{d\times d}$  
   \st  $ X = T A D $. 
 A list $X\subseteq \Z^{d} \oplus \Z_{q_1} \oplus \ldots \Z_{q_n}  $ is a \emph{quotient of a doubly scaled unimodular list} (QDSUL)
 if there is a doubly scaled unimodular list
 $\lift(X) \subseteq \Z^{(d + n)\times (N+n)}$    \st  $ X = \lift(X) / Y$, where 
 $Y\subseteq X$ denotes the set of columns  $d+1,\ldots, d+n$. By definition, these columns are unit vectors in the matrix $A$ that corresponds to the 
 doubly scaled unimodular list $\lift(X)$. 
 In other words, we can write
   \begin{equation}
   \label{eq:ShapeQDSUL}
 \lift(X) = T\cdot\,
\begin{blockarray}{rrr} %
  & Y &   \\
\begin{block}{(c|c|c)}
 I_d  &  0  & *    \\ \BAhline
 0  &  I_n  & *   \\
\end{block}
\end{blockarray}\,\cdot D,
\end{equation}
   where $I_j$ denotes a ($j\times j)$-identity matrix and $*$ denotes arbitrary matrices of  suitable dimensions.
  Note that while a QSUL is not always a QDSUL (the latter has the identity matrix requirement),
  every QSUL can be transformed into a QDSUL by a unimodular transformation.
  Unimodular transformations preserve   the arithmetic matroid structure.

\begin{Proposition}[strong/QSUL]%
\label{Proposition:RegularScaledUnimodular}
  Let $\Acal$ be a regular arithmetic matroid.
  Then the  following assertions are equivalent:
\begin{enumerate}[(i)]
 \item $\Acal$ is strongly multiplicative. If $\Acal$ is torsion-free, this means that all bases of $\Acal$ are multiplicative.
 If $\Acal$ has torsion, this means that it is the quotient of a torsion-free arithmetic matroid, all of whose bases are  multiplicative.
 \item $\Acal$ can be represented by a quotient of a scaled unimodular list (QSUL).
 \end{enumerate}
\end{Proposition}

\begin{Proposition}[weak/QDSUL]
\label{Proposition:RegularDoublyScaledUnimodular}
  Let $\Acal$ be a regular arithmetic matroid.
  Then the  following assertions are equivalent:
\begin{enumerate}[(i)]
 \item $\Acal$ is weakly multiplicative.
\item $\Acal$ can be represented by a quotient of a doubly scaled unimodular list (QDSUL).
\end{enumerate}
\end{Proposition}
 \begin{Remark}
     If we do not assume that the totally unimodular list in the definition of a QDSUL
     starts with an identity matrix, 
     Proposition~\ref{Proposition:RegularDoublyScaledUnimodular}  is false:
     The matrix $X_1$ below is a doubly scaled unimodular list, but it does not have a multiplicative basis.
     It is also important that the matrices $T$ and $D$ may have entries in $\Q\setminus \Z$. The matrix $X_2$ below represents an arithmetic matroid that is weakly multiplicative,
     but it cannot be obtained by scaling the rows and columns of a totally unimodular matrix by integers.
 \begin{equation}
  X_1 = \begin{pmatrix} 1 & 1 \\ 0 & 2 \end{pmatrix}
 \qquad
  X_2 = \begin{pmatrix}
          1 & 0 & 1 \\ 0 & 1 & 2
        \end{pmatrix}
 \end{equation}
\end{Remark}
 \begin{Remark}
 QSULs
 have appeared in the literature before.
 They have certain nice properties concerning combinatorial interpretations of their arithmetic Tutte polynomials \cite[Remark~13]{backman-lenz-2016}.

 In \cite{lenz-unique-2019}, the author showed that weakly multiplicative and torsion-free arithmetic matroids have 
 a unique representation (up to a unimodular transformation from the left and multiplying the columns by $-1$).
 Callegaro--Delucchi pointed out that this implies that the
 integer cohomology algebra of the corresponding centred toric arrangement is determined combinatorially, \ie by the arithmetic matroid \cite{callegaro-delucchi-2017}.
\end{Remark}

  The advantage of regular and strongly multiplicative arithmetic matroids over
  regular and weakly multiplicative arithmetic matroids 
  is that 
  their multiplicity function can be calculated very easily
  (Lemma~\ref{Lemma:MultiplicityRegularMultiplicative}).
  In Subsection~\ref{Subsection:LabelledGraphsResults} we will see that strongly multiplicative arithmetic matroids  arise naturally from 
  labelled graphs.
\begin{Question}
 There is a gap between 
 Theorem~\ref{Theorem:NonRegularNonRepresentable} and Theorem~\ref{Theorem:PowerOfArithmeticMatroid}.
 In particular, the case of regular arithmetic matroids that are not multiplicative is not covered.
 Example~\ref{Example:K3squareablenotmultiplicative} and
 Example~\ref{Example:NonSquareableC4}
 show
 that 
 Theorem~\ref{Theorem:PowerOfArithmeticMatroid} is not optimal.
 Is it possible to make a more precise statement about when 
 taking a power of the multiplicity function of an arithmetic matroid whose underlying matroid is regular preserves
 representability?
\end{Question}

 From our results, one can easily deduce the following new characterisation of regular matroids.
 This characterisation is mainly of theoretical interest as it does not yield an obvious
 algorithm to check regularity.
 In addition, a polynomial time algorithm to test regularity of a matroid that is given by an 
 independence oracle is already known \cite{truemper-1982}.

\begin{Corollary}
 \label{Corollary:RegularCharacterisation}
 Let $M$ be a matroid of rank $d$ on $N$ elements.
 Then the following assertions are equivalent:
 \begin{enumerate}[(i)]
  \item $M$ is a regular matroid.
  \item There is an ordered field $\K$ and a non-negative integer $k\neq 1$ \st
  there are matrices $X$ and $X_k$ with entries in $\K$ that represent $M$ and $ \abs{\Delta_I(X)}^k = \abs{\Delta_I(X_k)}$ for all $I\in \binom{[N]}{d}$.
  \item For every ordered field $\K$ and every non-negative integer $k\neq 1$ 
  and every representation $X$ over $\K$, there is a representation $X_k$ over $\K$
  \st $ \abs{\Delta_I(X)}^k = \abs{\Delta_I(X_k)}$ for all $I\in \binom{[N]}{d}$.
  \item $M$ has a representation over $\Z$ with corresponding arithmetic matroid $\Acal$, \st
  the arithmetic matroid $\Acal^k$ is representable for any non-negative integer $k\ge 0$.
 \end{enumerate}
\end{Corollary}
 Note that $(i)\Rightarrow(iv)$ follows by choosing a  totally unimodular matrix as a representation. The rest follows directly from
 Theorem~\ref{Theorem:NonRegularMatricesNewNew}, Theorem~\ref{Proposition:RegularMatrices}, and Theorem~\ref{Theorem:NonRegularNonRepresentable}.

\subsection{Two different multiplicity functions}
\label{Subsection:TwoDifferentMultiplicityFunctions}

 Delucchi--Moci \cite{delucchi-moci-2018}, as well as Backman and the author 
 \cite{backman-lenz-2016} 
 showed that if
 $\Acal_1=(E, \rank,m_1)$
 and
 $\Acal_2=(E, \rank,m_2)$ are arithmetic matroids with the same underlying matroid, then 
 $\Acal_{12} :=(E, \rank,m_1 m_2)$ is also an arithmetic matroid.
 It is a natural question to ask if representability of  $\Acal_1$ and $\Acal_2$ implies representability of $\Acal_{12}$.
 Can our previous results be generalized to this setting? 
 For 
 Theorem~\ref{Theorem:NonRegularMatricesNewNew}
 and Theorem~\ref{Theorem:NonRegularNonRepresentable}, this is false in general (Example~\ref{Example:NoTwoFunctionsNonRegular}).
 For Theorem~\ref{Theorem:PowerOfArithmeticMatroid}, it is false as well (Example~\ref{Example:ArithmeticMatroidsNoTwoFunctions}).
  On the other hand, the constructive result for Pl\"ucker vectors (Theorem~\ref{Proposition:RegularMatrices}) can 
  be generalised to the setting of two different Pl\"ucker vectors that define
 the same matroid  (Theorem~\ref{Proposition:MatricesTwoMultiplicities}).
 
\begin{Example}
\label{Example:NoTwoFunctionsNonRegular}
Let
\begin{align}
X_1 &=
\left(\begin{array}{rrrr}
 1 & 0 & -2 & -2 \\
 0 & 1 & 1  & -1 
\end{array}\right),
&
X_2 &=
\left(\begin{array}{rrrr}
 1 & 0 & -1 & -1 \\
 0 & 1 &  3 &  1 
\end{array}\right)
\displaybreak[2]
\\
\text{and }
 X_{12} &=
\left(\begin{array}{rrrr}
 1 & 0 & -2 & -2 \\
 0 & 1 &  3 & -1 
\end{array}\right).
\end{align}
All three matrices have the same non-regular underlying matroid: $U_{2,4}$.
The maximal minors of $X_{12}$ are $\Delta_{12}(X_{12}) = 1=1\cdot 1 = \Delta_{12}(X_1)\Delta_{12}(X_2)$, $\Delta_{13}=3=1\cdot 3$, $\Delta_{14}=-1=(-1)\cdot 1$, $\Delta_{23}=2= 2\cdot 1$,
$\Delta_{24}=2=2\cdot 1$, and $\Delta_{34}=8=4\cdot 2$. 
Thus, the Pl\"ucker vector of $X_{12}$ is the element-wise product of the Pl\"ucker vectors of $X_1$ and $X_2$.
The columns of each of the three matrices are primitive vectors, hence the multiplicity functions of the arithmetic matroids assumes the value $1$
on each singleton. This implies that $X_{12}$ represents the arithmetic matroid $\Acal_{12}=(E, \rank, m_1 m_2)$, where $m_1$ and $m_2$ denote the multiplicity functions
induced by $X_1$ and $X_2$, respectively.
\end{Example}

\begin{Example}
\label{Example:ArithmeticMatroidsNoTwoFunctions}
 Let  $X_1 = (2, 3)$ and $X_2 = (3,2) $ and let $\Acal_1=(\{x_1,x_2\}, \rank, m_1)$ and $\Acal_2=(\{x_1,x_2\}, \rank, m_2)$ denote the corresponding
 arithmetic matroids. 
 Both $\Acal_1$ and $\Acal_2$ are strongly multiplicative.
 Of course,
 \begin{equation}
 m_1(\{x_1,x_2\})= m_2(\{x_1,x_2\})=\gcd(2,3)=1.
 \end{equation}
 Let $m_{12} := m_1 \cdot m_2$. We have $m_{12}(\{1\}) = m_{12}(\{2\}) = 6$,
 and $m_{12}(\{x_1,x_2\}) = 1 \neq 6 = \gcd(m_{12}(\{1\}),m_{12}(\{2\}))$.
 By Lemma~\ref{Lemma:MultiplicityFreeGroup}, this implies that $\Acal_{12}=(\{x_1,x_2\},\rank, m_{12})$ is not representable. 
\end{Example}

\enlargethispage{1pt}
 
 Recall that two matroids $(E_1, \rank_1)$ and $(E_2, \rank_2)$ are isomorphic if there is a bijection
 $ f : E_1 \to E_2$ \st $\rank_1(S)= \rank_2(f(S))$ for all $S\subseteq E_1$.
 In the next theorem we will use the stronger notion of two matrices $X_1,X_2\in \K^{d\times N}$ defining the same 
 labelled matroid. This means that 
 $\Delta_I(X_1) = 0 $ if and only if $\Delta_I(X_2) = 0$ for all $ I \subseteq \binom{ [N] }{ d}$.
 Put differently, equality as labelled matroids means that two matroids on the same ground set are equal without permuting the elements.
 
 \begin{Theorem}[regular, Pl\"ucker]
 \label{Proposition:MatricesTwoMultiplicities}
 Let $\K$ be a field. Let $X_1,X_2\in \K^{d\times N}$ be two matrices that represent the same regular labelled matroid. 
 Then
  \begin{enumerate}[(i)]
    \item for any  $k_1, k_2\in \Z$, there is  $X_{k_1k_2}\in \K^{d\times N}$  
    \st for each $I\subseteq \binom{[N]}{d}$, 
    \begin{equation}
      \Delta_I( X_1 )^{ k_1 } \Delta_I(X_2)^{k_2} = \pm \Delta_I(X_{k_1k_2}).
    \end{equation}
    \item If $k_1+k_2$ is odd, then
     there is  $X_{k_1k_2}\in \K^{d\times N}$  \st the following stronger equalities hold
     for all $ I \subseteq \binom{ [N] }{ d}$:
    $\Delta_I( X_1 )^{ k_1 } \Delta_I(X_2)^{k_2} = \Delta_I(X_{k_1k_2})$.
    \item 
    If  $\K$ is an ordered field
    then for all $k_1, k_2$ \st the expression $x_1^{k_1}x_2^{k_2}$ is well-defined for all positive elements   $ x_1, x_2 \in \K $ (\eg $k_1,k_2\in \Z$ or $\K=\R$ and $k_1,k_2\in \R$), 
    there is $X_{k_1k_2}\in \K^{d\times N}$ 
    \st 
      \begin{equation}
        \abs{\Delta_I( X_1 )}^{ k_1 } \!\cdot \abs{\Delta_I(X_2)}^{k_2} = \abs{\Delta_I(X_{k_1k_2})}
      \end{equation}
      and $ \Delta_I( X_1 ) \cdot \Delta_I(X_2) \cdot {\Delta_I(X_{k_1k_2})} \ge 0$
      for all $I\subseteq \binom{[N]}{d}$.      
  \end{enumerate}
 Here, we are using the convention $0^0=0$.
 \end{Theorem}

\subsection{Arithmetic matroids defined by labelled graphs}
\label{Subsection:LabelledGraphsResults}

   In this subsection
   we will  consider arithmetic matroids defined by labelled graphs that were introduced by D'Adderio--Moci \cite{moci-adderio-flow-2013}.
   This is a rather simple class of arithmetic matroids whose multiplicity function is strongly multiplicative.
   In the case of labelled graphs, one can make 
  Theorem~\ref{Theorem:PowerOfArithmeticMatroid}
  very explicit.

   A \emph{labelled graph} is a graph $\Gcal=(V,E)$ together with a labelling $ \ell : E \to \Z_{\ge 1} $.
   The graph $\Gcal$ is allowed to have multiple edges, but no loops.
   The set of edges $E$ is partitioned into a set $R$ of \emph{regular edges} and a set $W$ of \emph{dotted edges}.
   The construction of the arithmetic matroid extends the usual construction of the matrix representation of a graphic matroid
   by the oriented incidence matrix:
   let $ V = \{v_1,\ldots, v_n\}$. We fix an arbitrary orientation $\theta$ \st each edge $e\in E$ can be identified with an ordered pair $(v_i,v_j)$.
   To each edge $e = ( v_i , v_j )$, we associate the element $x_e \in \Z^n$ defined as the vector
   whose $i$th coordinate is $-\ell(e)$ and  whose $j$th coordinate is $\ell(e)$.
    Then we define the two lists $X_R = (x_e)_{e\in R}$ and $ X_W = (x_e)_{e\in W} $.
   We define $ G = \Z^n / \langle \{ x_e : e \in W \} \rangle$ and we denote by $\Acal(\Gcal,\ell)$
   the arithmetic matroid defined by the 
   list $X_R$ in $G$.

 \begin{Proposition}
  \label{Proposition:LabelledGraphsSquareable}
    Let  $(\Gcal,\ell)$ be a labelled graph and let $\Acal(\Gcal,\ell)$ be the corresponding arithmetic matroid. Let $k\ge 0$ be an integer. Then
    $\Acal(\Gcal,\ell)^k =  \Acal(\Gcal,\ell^k)$, where $\ell^k(e):= (\ell(e))^k$ for each edge $e$. 

\end{Proposition}

\begin{Proposition}
\label{Proposition:LabelledGraphRepresentableRegularStronglyMultiplicative}
   Arithmetic matroids defined by a labelled graph are regular  and strongly multiplicative.
\end{Proposition}

\begin{Remark}
  \label{Remark:LabelledGraphUnderlyingGraph}
    Note that the class of arithmetic matroids defined by a labelled graph is a 
    subset of the class of arithmetic matroids whose underlying matroid is graphic:
    let  $(\Gcal, \ell)$ be a labelled graph and let $\Acal(\Gcal,\ell)$ be the corresponding arithmetic matroid.
    Its underlying matroid is the graphical matroid of the graph obtained from $\Gcal$ by contracting the dotted edges.
    But in general, arithmetic matroids with an underlying graphic matroid do not arise from a labelled graph   
    (see Example~\ref{Example:GraphicNotLabelled}).
\end{Remark}

\subsection{Examples}
\label{Subsection:Examples}
In this subsection we will present some further examples.

 \begin{Example}
 \label{Example:GraphicNotLabelled}
  \begin{align}
  \text{Let }
    X &= \begin{pmatrix}
           1 & 0 & 1 \\
           0 & 3 & -2 \\
          \end{pmatrix} = 
          \begin{pmatrix}
             1 & 0 \\ 0 & -2 
          \end{pmatrix}
          \begin{pmatrix}
           1 & 0 & 1 \\
           0 & 1 & 1 \\
          \end{pmatrix} 
          \diag(1,-\frac 32, 1)
\\
        \quad \text{and} \quad 
    X_k &= \begin{pmatrix}
           1 & 0 & 1 \\
           0 & 3^k & - 2^k \\
          \end{pmatrix} =   
          \begin{pmatrix}
             1 & 0 \\ 0 & -2 
          \end{pmatrix}^k
          \begin{pmatrix}
           1 & 0 & 1 \\
           0 & 1 & 1 \\
          \end{pmatrix} 
          \diag(1,-\frac 32, 1)^k.
  \end{align}
On the right-hand side of the equations is the decomposition of $X$ and $X_k$ as 
a  doubly scaled unimodular list.
See the proof of Proposition~\ref{Proposition:RegularDoublyScaledUnimodular} 
and Remark~\ref{Remark:RegularDecomposition} for information on how this decomposition can be obtained.

Note that the first two columns of $X$ form a diagonal matrix. 
Hence the arithmetic matroid $\Acal(X)$ is weakly
multiplicative.
 The matrix $X_k$ represents the arithmetic matroid $\Acal(X)^k$.
 The underlying matroid is graphic, it is defined by the complete graph $K_3$. 
 However,
 since $m(x_1,x_3)\neq m(x_1)m(x_3)$, $\Acal(X)$ is not strongly multiplicative (here, $x_i$ denotes the $i$th column of $X$).
  Hence by  Proposition~\ref{Proposition:LabelledGraphRepresentableRegularStronglyMultiplicative},
 the arithmetic matroid 
 does not arise from a labelled graph.
\end{Example}

 The next two examples show that our results are not optimal.
 We present two matrices $X$ that do not satisfy the conditions of 
 Theorem~\ref{Theorem:PowerOfArithmeticMatroid},
 yet the arithmetic matroid $\Acal(X)^2$ is representable.
 The underlying matroids are $U_{2,3}$ (corresponds to the complete graph $K_3)$ and $U_{3,4}$ (corresponds to the cycle graph $C_4$).

\begin{Example}
 \label{Example:K3squareablenotmultiplicative}
  \begin{equation}
   \text{Let }   X = \begin{pmatrix} 1 & 1 & 1 \\ 0 & 2 & 4 \end{pmatrix} 
   \:\text{ and }\:
    X_2 = \begin{pmatrix} 1 & 1 & 3 \\ 0 & 4 & 16 \end{pmatrix}.
  \end{equation}
 It is easy to check that $\Acal(X)^2 = \Acal(X_2)$ holds.
 All singletons have multiplicity $1$, but all bases have a higher multiplicity.
 Hence $\Acal(X)$ is not multiplicative.
\end{Example}

 \begin{Example}
\label{Example:NonSquareableC4}
Let
\begin{align*}
    X =
  \begin{pmatrix}
    1 & 1 & 2 & 1  \\
    0 & 2 & 1 & 2  \\
    0 & 0 & 3 & 2
  \end{pmatrix}
 &=  
\left(\begin{array}{rrr}
1 & -\frac{2}{3} & -\frac{4}{3} \\
0 & -\frac{4}{3} & -\frac{2}{3} \\
0 & 0 & -2
\end{array}\right)
 \left(\begin{array}{rrrr}
1 & 0 & 0 & 1 \\
0 & 1 & 0 & 1 \\
0 & 0 & 1 & 1
\end{array}\right)
%
%
\cdot \diag(1,-\frac 32, -\frac 32, -1).
\end{align*}

 On the right-hand side of the equation is the decomposition of $X$ 
as in Corollary~\ref{Corollary:BLrepresentationOfRegularMatroid} and Remark~\ref{Remark:RegularDecomposition}.
 The maximal minors are $\Delta_{123}=6$, $\Delta_{124}=4$, $\Delta_{134}=-4$, and $\Delta_{234}=-6$.
Using the construction that appears in the second part of the proof of Theorem~\ref{Proposition:MatricesTwoMultiplicities}, we obtain
the matrix $X_2$ that represents the squared Pl\"ucker vector of $X$ (cf.~Theorem~\ref{Proposition:RegularMatrices}):
\begin{align}
X_2 &=
\left(\begin{array}{crcc}
\frac{64}{9} & 0 & 0 &  -\frac{64}{9} \\
0 & -\frac{9}{4} & 0 & -1 \\
0 & 0 & -\frac{9}{4} & -1
\end{array}\right) \\
&= 
\left(\begin{array}{ccc}
\frac{64}{9} & 0 & 0 \\
0 & 1 & 0 \\
0 & 0 & 1
\end{array}\right)
\left(\begin{array}{rrrr}
1 & 0 & 0 & 1 \\
0 & 1 & 0 & 1 \\
0 & 0 & 1 & 1
\end{array}\right)
\diag(1,-\frac 94, -\frac 94, -1).
\end{align}

 Since the arithmetic matroid $\Acal(X)$ is not weakly multiplicative, 
 we do not know a general method to find a representation of $\Acal(X)^2$.
 Nevertheless, in this case, there is one:
\begin{equation}
 X_2 =
\begin{pmatrix}
 1 & 1 & 1 & 1 \\
 0 & 4 & 4 & 0 \\
 0 & 0 & 9 & 4 
\end{pmatrix}.
\end{equation}

\end{Example}

\subsection*{Where to find the proofs} 
 
 Before reading the proofs, 
 the reader should make sure that he or she is familiar with the background 
 material that is explained in Section~\ref{Section:Background}.
 In Section~\ref{Section:Proofs}  we will prove 
 Theorem~\ref{Theorem:NonRegularMatricesNewNew} and
  Theorem~\ref{Theorem:NonRegularNonRepresentable}, the two results on non-regular matroids.
 In Section~\ref{Section:SupplementaryResults}  we will prove %
 Proposition~\ref{Proposition:NonRegularMatricesNewNewBestPossible}.  
 Its proof is fairly elementary and does not require much background.
 In Section~\ref{Section:RegularMatroids}  we will first discuss the fact that regular matroids have a unique representation,
 which is a crucial ingredient of the remaining proofs. Then we will prove
 Theorem~\ref{Proposition:MatricesTwoMultiplicities}.
 In Section~\ref{Section:MultiplicityFunctions}  we will prove some lemmas about representable arithmetic matroids and their multiplicity functions.
 These lemmas will be used in Section~\ref{Section:MultiplicativeRegularAriMatroids}
  in the proofs of   Proposition~\ref{Proposition:RegularScaledUnimodular},
 Proposition~\ref{Proposition:RegularDoublyScaledUnimodular}, and  
 Theorem~\ref{Theorem:PowerOfArithmeticMatroid}.	 
  Proposition~\ref{Proposition:LabelledGraphsSquareable} and
 Proposition~\ref{Proposition:LabelledGraphRepresentableRegularStronglyMultiplicative}
 will be proven in the second part of 
 this section.

 Theorem~\ref{Proposition:RegularMatrices} 
 is a special case of  Theorem~\ref{Proposition:MatricesTwoMultiplicities} ($k_1=k$, $k_2=0$, $X_1=X_2=X$)
 and 
 therefore does not require a proof.
 Theorem~\ref{Theorem:NonRegularMatricesNew} 
 is a combination of special cases of
 Theorem~\ref{Theorem:NonRegularMatricesNewNew} and 
 Theorem~\ref{Proposition:RegularMatrices}.

\section{Subvarieties of the Grassmannian}
\label{Section:SubvarietiesGrassmannian}

In the first subsection we will recall some facts about the Grassmannian and antisymmetric tensors.
In the second subsection we will define the regular Grassmannian 
as the set of all points in the Grassmannian that define a regular matroid.
We will show that it is a projective subvariety of the Grassmannian.
In the third subsection we will compare the  regular Grassmannian with the variety of Pl\"ucker vectors that have a $k$th power.

\subsection{Antisymmetric tensors and the Grassmannian}
\label{Subsection:GrassmannianIntro}

In this subsection we will recall some facts about the Grassmannian and antisymmetric tensors from
\cite[Section~2.4]{OrientedMatroidsBook} (see also  \cite[Chapter~3.1]{gelfand-kapranov-zelevinsky-1994}).
We will assume that the reader is familiar with elementary algebraic geometry (see \eg \cite{cox-little-oshea-2015,harris-1995,shafarevich-2013}).
Let us fix a field $\K$ and integers $0 \le d\le N$.
As usual, $\Lambda^d \K^N$ denotes the $d$-fold exterior product of the vector space $\K^N$.
The elements of $\Lambda^d \K^N$ are called \emph{antisymmetric tensors}. The space 
$\Lambda^d \K^N$ is an $\binom{N}{d}$-dimensional $\K$-vector space with the canonical basis
\begin{equation}
  \{  e_{i_1} \wedge \ldots \wedge e_{i_d} : 1 \le i_1 < \ldots  < i_d \le N           \}.
\end{equation}
 For $i_1,\ldots, i_d \subseteq [N]$, there is a function $[ i_1\ldots i_d ] \in (\Lambda^d \K^N)^*$
 that we call the 
 \emph{bracket}.
  If $i_1< \ldots < i_d$, $[ i_1\ldots i_d ]$ is the coordinate function.
  If two of the indices are equal, the bracket is equal to $0$. Furthermore, 
  identities such as
 $[i_1 i_2 i_3\ldots  i_d] = - [i_2 i_1 i_3\ldots  i_d]$ hold.
 As $\K^N$ has a canonical basis, we can canonically identify $\K^N$ and its dual space $(\K^N)^*$ and thus also $(\Lambda^d \K^N)^*$
 and $(\Lambda^d \K^N)$.
 The ring of polynomial functions on $\Lambda^d \K^N$ is the \emph{bracket ring} 
\begin{equation}
  \sym(\Lambda^d \K^N) = \K\left[\, \{ [i_1 i_2 \ldots i_d ] : 1 \le i_1 < \ldots < i_d \le N \} \,\right]. 
\end{equation}
 Of course, the bracket ring is canonically isomorphic to the polynomial
 ring  
 \begin{equation}
 \K[ m_{i_1 i_2 \ldots i_d} : 1 \le i_1 < \ldots < i_d \le N  ].
 \end{equation}
 The isomorphism maps  $[i_1\ldots  i_d]$  to $m_{i_1 i_2 \ldots i_d}$.

 An antisymmetric tensor $\xi\in \Lambda^d \K^N$ is called \emph{decomposable} if it is non-zero and it can be written as
 $\xi = v_1 \wedge \ldots \wedge v_d$ for some vectors $v_1,\ldots, v_d\in \K^N$.
 Each decomposable tensor $\xi = v_1\wedge \ldots \wedge v_d \in \Lambda^d \K^N$ defines 
  the $d$-dimensional space $\spa(v_1\, \ldots,  v_d) \subseteq \K^N$ and for $c\in \K^*$, $c\xi$ and $\xi$ define the same space. 
  In fact, the Grassmannian is equal to the set of decomposable tensors modulo scaling:
 \begin{equation}
 \label{eq:GrassmannianTensors}
  \Gr_\K(d,N) = \{ \bar \xi \in \Lambda^d \K^N / \K^* : \xi \text{ decomposable}   \}.
\end{equation}
 
 The Grassmannian can also be written in a different way:
 the points in $\Gr_\K(d,N)$ are in one-to-one correspondence with 
 the set of $(d\times N)$-matrices of full rank $d$
 modulo a left action of  $\GL(d,\K)$.
 Here, a matrix $X$ with row vectors  $v_1,\ldots, v_d \in \K^N$ is mapped  to the decomposable tensor $v_1\wedge \ldots \wedge v_d$.
 Since the matroid represented by a matrix $X$ is also invariant under 
 a left action of  $\GL(d,\K)$, this leads to a stratification of the  Grassmannian $\Gr_\K(d,N)$ 
 into the matroid strata  $\Rcal(\Mcal) = \{ \bar X \in \K^{d\times N} / \GL(d,\K) : X \text{ represents } \Mcal\}$ \cite{gelfand-gorseky-macpherson-serganova-1987}.

 Let $\xi = v_1 \wedge \ldots \wedge v_d$ be a decomposable tensor. 
 Let $X$ denote the matrix whose rows are the vectors $v_1,\ldots, v_d$.
 Then
 \begin{equation}
  \xi = \sum_{ \substack{I = \{ i_1,\ldots, i_d\}\\1\le i_1 < \ldots < i_d \le N} }  \Delta_I(X) \, e_{i_1} \wedge \ldots \wedge e_{i_d},
 \end{equation}
 where 
 for  $I = \{ i_1, \ldots ,i_d\}$  ($1\le i_1 < \ldots < i_d \le N$),
 $\Delta_I(X)$ denotes
  a \emph{Pl\"ucker coordinate} of $X$, \ie the determinant of the 
 $(d\times d)$-submatrix of $X$ consisting of the columns $i_1,\ldots, i_d$.
 The Pl\"ucker coordinates can also be obtained by evaluating the brackets at $\xi$: $\Delta_{I}(X) = [i_1\ldots i_d] ( \xi )$.
 Using \eqref{eq:GrassmannianTensors}, one can see that the homogeneous Pl\"ucker coordinates
 induce an embedding of $\Gr_\K(d,N)$ into $\bigl(\binom{N}{d}-1\bigr)$-dimensional projective space $\Lambda^d \K^N / \K^*$.
 In fact, the Grassmannian is an irreducible subvariety of this space.

\begin{Theorem}[Grassmann--Pl\"ucker relations]
\label{Theorem:GPrelations}
 Let $\K$ be a field of characteristic $0$.
 Then the Grassmannian $\Gr_\K(d,N)$, embedded in $\Lambda^d \K^N / \K^*$
 is the zero set of 
 the quadratic polynomials 
  \begin{equation}
  \label{eq:GP}
   [ b_1 b_2 b_3 \ldots b_d] [ b_1'b_2'\ldots b_d'] - \sum_{i=1}^d [b_i'b_2b_3\ldots, b_d] [ b_1' \ldots b_{i-1}' b_1 b'_{i+1} \ldots b_d']   
 \end{equation}
where $b_1,\ldots, b_d, b_1',\ldots, b_d' \in [N]$.
 \end{Theorem}
 The equations defined by the polynomials in $\eqref{eq:GP}$ are called the \emph{Grassmann--Pl\"ucker relations}.
 The 
 ideal $I_{\Gr_\K(d,N)} \subseteq \sym(\Lambda^d \K^N)$ that is generated by these polynomials %
 is called
 the \emph{Grassmann--Pl\"ucker ideal}.

\subsection{The regular Grassmannian}
 In this subsection we will define the regular Grassmannian as the union of all the
 regular matroid strata of the Grassmannian.
 We will show that it is a projective subvariety.

\begin{Definition}[Regular Grassmannian]
  Let $\K$ be a field and let  $0 \le d\le N$ be integers.
  Then we define the \emph{regular Grassmannian} $\RGr_\K(d,N)$ as  
   \begin{equation*}
   \RGr_\K(d,N):= \{ \bar X \in \Gr_\K(d,N) :  {X}  \text{ represents a regular matroid}  \}.
 \end{equation*}
\end{Definition}
 Instead of considering the regular Grassmannian  $\RGr_\K(d,N)\subseteq \Lambda^d\K^N/\K^*$,
 one can of course  also consider the set of  decomposable antisymmetric tensors that define a regular matroid.
 This is a subset of $\Lambda^d\K^N$. Since this subset is invariant under scaling by elements in $\K^*$
 and this operation does not change the matroid, these two points of view are equivalent for our purposes.

\begin{Definition}
  Let $\K$ be a field and let  $0 \le d\le N$ be integers.
 We define the following two ideals in the bracket ring $(\Lambda^d \K^N)^*$:
 \begin{equation*}
   R_{d,N} :=\left(  \prod_{(\alpha,\beta) \in \binom{\{p_1,\ldots, p_4\}}{2}} [{b_1 \ldots b_{d-2} \alpha \beta}]: 
                      \{ b_1,\ldots, b_{ d-2 }, p_1, p_2, p_3, p_4 \} \in \binom{[N]}{d+2} \right)
\end{equation*}
and    $I_{\RGr(d,N)} := I_{\Gr(d,N)} +  R_{d,N}$. 
\end{Definition}

\begin{Proposition}
 \label{Proposition:IdealRegularGrassmannian}
 Let $\K$ be a field of characteristic $0$ and let  $0 \le d\le N$ be integers.
  Then the  regular Grassmannian $\RGr_\K(d,N)$  is a projective variety. It is the zero set of 
  $ I_{\RGr(d,N)} $.
 In particular, the regular Grassmannian is a variety that is defined by homogeneous quadratic polynomials and monomials of degree $6$.
\end{Proposition}

\begin{proof}
 As usual, let $V(I)$ denote the variety defined by the ideal $I$.
 It is a basic fact that for two ideals $I_1$ and $I_2$, $V(I_1 + I_2)=V(I_1) \cap V(I_2)$ holds.
 Hence $V( I_{\RGr(d,N)} + R_{d,N} ) = V( I_{\Gr(d,N)} )\cap V( R_{d,N} )$.
 This implies that  it is sufficient to prove that a decomposable tensor $\xi$
 that represents a point  in the Grassmannian $\Gr_\K(d,N)$  defines a regular matroid if and only if
 it satisfies the monomial equations given by $R_{d, N}$.

 To simplify notation, we will work in the setting of decomposable antisymmetric tensors.
 Let $\xi=v_1\wedge \ldots \wedge v_d \in \Lambda^d \K^N$ be such a tensor 
 and let $\Mcal_\xi = ([N], \rank)$ denote the corresponding matroid,
 \ie the matroid on $N$ elements that is defined by the matrix $X$ whose rows are the vectors $v_1,\ldots, v_d$.
 We are working over a field of characteristic $0$, hence being regular is equivalent to not having a $U_{2,4}$ minor (see Corollary~\ref{Corollary:RegularExcludedMinors}).
 By construction, $\Mcal_\xi$ has rank $d$. 
 All of its rank two minors with four elements can be constructed as follows:
 pick an independent set $ A_2\subseteq [N] $ of cardinality $d-2$ and a set $A_1 \subseteq [N] \setminus A_2$ 
 of cardinality $4$ \st $A_1 \cup A_2$ contains a basis. Then consider
 $(\Mcal_\xi / A_2)|_{A_1}$. This is a matroid on the ground set $A_1$ that is represented by the matrix
 $(X / A_2)|_{A_1}$.
 This minor is \emph{not} isomorphic to $U_{2,4}$ if at least one of its six Pl\"ucker coordinates is equal to $0$.
 Recall that $(X / A_2)$ denotes the matrix obtained from $X$ by deleting the columns corresponding to $A_2$ and then projecting 
 the remaining columns to the quotient of $\K^d$ by the space  spanned by the columns indexed by $A_2$.
 There is $c\in \K^*$ \st
 $ c \cdot \Delta_I( (X / A_2)|_{A_1} ) = \Delta_{A_2\cup I}(X)$ for all $I\in \binom{A_1}{2}$.
 This follows from the fact that after applying a transformation $ T \in \SL(d,\K)$, we have
 \begin{equation}
 TX =
\begin{blockarray}{rrr} %
 A_1 & A_2 &  &  \\
\begin{block}{(r|r|r)}
 X_1  &  0  & *    \\  \BAhline
 *  &  D  & *   \\
\end{block}
\end{blockarray}\,,
\end{equation}
 where $X_1\in \K^{2\times 4}$, $D\in \K^{(d-2) \times (d-2)}$ is a diagonal matrix and $*$ denotes arbitrary matrices of suitable dimensions.
 Then $c = \det(D)$.
 We can deduce that $\Mcal_\xi$ is regular if and only if 
 for all pairs $(A_1,A_2)$, where $A_2 \in \binom{[N]}{d-2}$ is an independent set in $\Mcal_\xi$  and
 $A_1 \in \binom{[N] \setminus A_2}{4}$, 
 $\prod_{ I \in \binom{A_1}{2}} \Delta_{A_2\cup I}(X) = 0$ holds.

 Recall that  
 if $A_2=\{b_1,\ldots, b_{d-2}\}$ and $I=\{p_i, p_j\}$ then 
 $  [ b_1\ldots b_d p_i p_j    ](\xi) = \pm \Delta_{A_2\cup I}(X)  $ holds.
  If $A_2$ is dependent, then $[ b_1\ldots b_d p_i p_j    ](\xi) = 0 $.
 Now it is clear that  $\Mcal_\xi$ has a $U_{2,4}$ minor  if and only if
 it satisfies the monomial equations given by $R_{d, N}$.
\end{proof}

 \begin{Example}
   Let us consider the case $N=4$, $d=2$.
   Then
  \begin{equation}
   I_{\RGr(2,4)} = ( [ 1 2 ] [ 3 4 ] +  [  2 3 ] [ 1 4 ] -  [ 2 4 ] [ 1 3 ],\, [1 2][1 3][1 4][2 3][2 4][3 4] ). 
  \end{equation}
  Hence $\RGr(2,4) = \bigcup_{ 1 \le \nu < \mu \le 4   }   \Gr(2,4) \cap V([i_\nu i_\mu])$, \ie the regular Grassmannian $\RGr(2,4)$, parametrised through
  Pl\"ucker coordinates, can be written as the union
  of the six intersections of the Grassmannian with each of the coordinate hyperplanes.
  In particular, the regular Grassmannian is in general not irreducible.
 \end{Example}

 \begin{Remark}
  For $k\in \R$,  let $f_k : \RGr_\R(d,N) \to \RGr_\R(d,N)$ denote the map that sends a Pl\"ucker vector to its $k$th power,
  while keeping the signs.
  It follows from Theorem~\ref{Proposition:RegularMatrices}.\eqref{enumerate:RegPlueckerOrderedField} that the image of $f_k$ is indeed contained in $\RGr_\R(d,N)$.
  Furthermore, for $k\neq 0$, the map is invertible and the inverse is given by $(f_k)^{-1}= f_\frac{1}{k}$.
  For any $k$, the map is continuous if we equip $\RGr_\R(d,N)$ with the topology that is induced by the Euclidean topology on $\R$.
  Hence for $k\neq 0$, $f_k : \RGr_\R(d,N) \to \RGr_\R(d,N)$ is a homeomorphism that leaves matroid strata invariant.
 \end{Remark}

 \subsection{The variety of Pl\"ucker vectors that have a $k$th power}

 In this subsection we will study the subvariety of the Grassmannian
 that consists of all points for which the element-wise $k$th power of its Pl\"ucker vector 
 is again a Pl\"ucker vector  (up to sign). 
 We will work over $\CC$ since one of the proofs requires an algebraically closed field.

 Let us first recall some facts from algebraic geometry.
 Let $X,Y\subseteq \CC^n$ be two subvarieties. 
 Then a \emph{regular map} $f: X\to Y$ is  the restriction of a polynomial map $f : \CC^n\to \CC^n$.
 Now suppose $X=Y=\CC^n$. Then the coordinate rings are $\CC[X]= \CC[x_1,\ldots, x_n]$ and $\CC[Y]= \CC[y_1,\ldots, y_n]$.
 For an integer $k\ge 1$ and $\sigma\in \{-1,+1\}^n$, we 
  define the regular map 
 \begin{equation}
   f_{k,\sigma}: X\to Y, \qquad  (p_1,\ldots, p_n) \mapsto (\sigma_1 p_1^k,\ldots, \sigma_n p_n^k) \in Y = \CC^n.
 \end{equation}
 This induces 
 \begin{equation}
 f_{\sigma,k}^\# : \CC[Y] \to \CC[X], \qquad  p \mapsto p\circ f_{\sigma, k}, 
 \end{equation}
 \ie   $y_i \mapsto \sigma_i x_i^k \in \CC[X] = \CC[x_1,\ldots, x_n] $.

 In Proposition~\ref{Proposition:SquareableVariety}
 we will determine the ideal of the subvariety of the Grassmannian that consists of all points for which the $k$th power of its Pl\"ucker vector 
 is again a Pl\"ucker vector (up to sign). This is the set
 \begin{equation}
     \Gr_\CC(d,N)  \cap \bigcup_{\sigma} f_{\sigma, k}(\Gr_\CC(d,N)), 
 \end{equation}
 where $\sigma \in \{-1,1\}^{\binom{[N]}{d}}$ runs over all possible choices of signs. 
 The ideal that defines $f_{\sigma, k}(\Gr_\CC(d,N))$ can be calculated using the following 
 well-known\footnote{See for example Proposition~2.2.1 in \cite{kraft-2014}.}
 lemma.
\begin{Lemma}
\label{Lemma:MorphismIdealClosure}
  Let $X,Y$ be affine varieties and 
  let $f : X \to Y$ be a regular map.
  Let $A = V(I)\subseteq X$ be the subvariety defined by the ideal $I$.
  Then the closure of the image  $f(A)$ is defined by $(f^\#)^{-1}(I)$.
 \end{Lemma}

 \begin{Definition}
   Let $k\ge 1$ and let $0 \le d \le N$ be  integers and let $n:= \binom{N}{d}$. 
   Let $I_{\Gr(d,N)} \subseteq \CC[x_1,\ldots, x_n]$ denote the Grassmann--Pl\"ucker ideal.
   Then we define  the ideal 
 \begin{equation}
    I_{k,\Gr(d,N))} = \prod_{ \sigma  } \left( ( f_{\sigma, k}^\#  )^{-1} (   I_{\Gr(d,N)} )  + I_{\Gr(d,N)} \right) \subseteq \CC[x_1,\ldots, x_n] ,
 \end{equation}
   where $ \sigma $ runs over $ \{ -1, +1 \}^n$.
 \end{Definition}

 \begin{Proposition}
  \label{Proposition:SquareableVariety}
   Let $k\ge 2$ and $0 \le d \le N$ be  integers. 
   Then the set 
   \begin{equation}
      \label{eq:VarietySquareable}
        \{ \xi \in \Lambda^d \CC^N :\text{$\xi$ and the $k$th power of $\xi$ up to sign are decomposable}    \} 
   \end{equation}
   is
   a (potentially reducible) subvariety of $\Lambda^d\CC^N$ that is defined by the ideal 
   \begin{equation*}
        I_{k,\Gr(d,N)}.    
   \end{equation*}
   The condition that the $k$th power of $\xi$ up to sign is decomposable 
   means that there is $\xi_k = v_1' \wedge \ldots \wedge  v_d'$ \st the Pl\"ucker coordinates satisfy 
   $\Delta_I(\xi)^k = \pm {\Delta_I(\xi_k)}$.
 \end{Proposition}
 Using Theorem~\ref{Theorem:NonRegularMatricesNew}, we can deduce the following corollary.
 \begin{Corollary}
  Let $0 \le d \le N$ and $k\ge 2$ be integers. %
  Then 
 \begin{equation}
     \RGr_\R(d,N) =  V_\CC( I_{k,\Gr(d,N )} )  \cap ( \Lambda^d \R^N / \R^*), 
 \end{equation}
  \ie the real regular Grassmannian is equal to the real part of the variety of complex decomposable tensors whose Pl\"ucker coordinates have a $k$th power, modulo scaling.
 \end{Corollary}
 \begin{proof}
  The  ``$\subseteq$'' part follows directly from Theorem~\ref{Theorem:NonRegularMatricesNew}.\\
  ``$\supseteq$'': let  $ (\xi_I)_{I\subseteq \binom{[N]}{d}} \in V_\CC( I_{k,\Gr(d,N )} )  \cap ( \Lambda^d \R^N / \R^*)$.
  This implies that $\xi$ has  real entries and there are matrices $ X, X_k \in \CC^{d\times N}$ \st
  $ \xi_I = \Delta_I(X)$ and $(\xi_I)^k = \pm\Delta_I(X_k)$ for all $I\in \binom{[N]}{d}$. 
  Since the Pl\"ucker vectors of $X$ and $X_k$ are real and they satisfy the Grassmann--Pl\"ucker relations,
  by Theorem~\ref{Theorem:GPrelations} it is possible to find real matrices $X'$ and $X_k'$ with the same properties.
   Using Theorem~\ref{Theorem:NonRegularMatricesNew}, we can now deduce that $\xi\in \RGr_\R(d,N)$.
 \end{proof}

 \begin{Lemma}
    \label{Lemma:SquaringClosed}
    The map $f_{\sigma, k} : \CC^n \to \CC^n$ is closed, \ie it maps closed sets to closed sets. 
    Here, we consider the Euclidean topology on $\CC^n$.  
 \end{Lemma}
 \begin{proof}
  Since every proper map is closed, it is sufficient
  to prove that $f_{\sigma, k}$ is proper, \ie 
  the preimage of every compact set in $\CC^n$ is compact.
  This is easy to see: let $C\subseteq \CC^n$ be compact, \ie it is closed and bounded.
  The preimage of $C$ is clearly bounded and since $f_{\sigma, k}$ is continuous, it is also closed. 
 \end{proof}

 \begin{Lemma} 
   \label{Lemma:SquaringZariskiClosed}
   $ f_{\sigma, k}( \Gr_\CC(d,N) )$ is an irreducible subvariety of $\Lambda^d \CC^N / \CC^* $.
 \end{Lemma}
 \begin{proof}
     Since      $f_{\sigma, k}$ is continuous in the Zariski topology, the image of an irreducible set is again irreducible.

     Recall that a constructible set is a finite union of locally closed subsets.
     Since $\CC$ is algebraically closed, 
     $ f_{\sigma, k}( \Gr_\CC(d,N) )$ is a constructible set
     by Chevalley's theorem (\eg \cite[Corollary~I.\S8.2]{mumford-1999}). 
     By Lemma~\ref{Lemma:SquaringClosed}, $ f_{\sigma, k}( \Gr_\CC(d,N) )$ is closed in the Euclidean topology.
     But for constructible subsets of $\CC^n$, the Zariski closure is the same as the closure in the Euclidean topology.
     Hence the set is also Zariski-closed.
 \end{proof}

 \begin{Remark}
  Over arbitrary fields, the image of $f_{\sigma, k}$ is in general not Zariski-closed.  
  For example, let us consider the variety  $V$ of antisymmetric decomposable tensors in $\Lambda^1 \R^1=\R^1$.
  Of course, $V=\R^1$.
  Then $f_{+,2}( V ) = \R_{\ge 0}$. 
 \end{Remark}

 \begin{Lemma}
  \label{Lemma:SquaringImageVariety}
  Let $Z\subseteq \CC^n$ be a Zariski-closed set that is defined by an ideal $ I \subseteq \CC[ x_1,\ldots,x_n ]$.
  Then
  \begin{equation}
     f_{\sigma, k}( Z ) =  V \left( (   f_{\sigma, k}^\#)^{-1} ( I ) \right).
  \end{equation} 
  In particular,   $ f_{\sigma, k}( \Gr_\CC(d,N) ) =  V \big(   (f_{\sigma, k}^\# )^{-1} (   I_{\Gr(d,N)}  ) \big) $.
  Here, $\Gr_\CC(d,N)\subseteq \Lambda^d \CC^N$ denotes the affine variety defined by the Grassmann--Pl\"ucker ideal.
 \end{Lemma}
 \begin{proof}
  This follows by combining Lemma~\ref{Lemma:MorphismIdealClosure} and Lemma~\ref{Lemma:SquaringZariskiClosed}.
 \end{proof}

\begin{proof}[Proof of Proposition~\ref{Proposition:SquareableVariety}]
   Recall that for two ideals $I$ and $J$, $V(I+J) = V(I)\cap V(J)$ and $V( I \cdot  J) = V( I) \cup V(J)$ holds.
   Let us fix $\sigma\in \{-1,+1\}^{\binom{[N]}{d}}$. 
   Then by Lemma~\ref{Lemma:SquaringImageVariety},
  \begin{equation}
       f_{\sigma,k} ( \Gr(d,N) ) = \{ \sigma \xi^k :  \xi \in \Gr(d,N)  \} = V( (f^\#_{\sigma,k})^{-1} ( I_{\Gr(d,N)} ) ),
  \end{equation}
  where $\sigma \xi^k := ( \sigma_I \cdot (\xi_I)^k    )_{ I \subseteq \binom{[N]}{d} }$. 
         \begin{align}
       \text{Hence }
       \{  \xi \in  \Gr(d,N) :  \sigma \xi^k \in \Gr(d,N)  \}  & = \{ \sigma \xi^k :  \xi \in \Gr(d,N)  \} \cap \Gr(d,N)
       \\ &
       \notag   
       =  V\left( (f^\#_{\sigma,k})^{-1} ( I_{\Gr(d,N)} ) + I_{\Gr(d,N)} \right).
  \end{align}
  From this we can deduce that the set that we are interested in, 
    $ \Gr_\CC(d,N) \cap \bigcup_\sigma f_{\sigma,k} ( \Gr_\CC(d,N) ) $, is defined by the given ideal.
\end{proof}

\begin{Example}
 Let us calculate $I_{2,\Gr(2,4)}$ using Singular \cite{singular-410}.
 The Grassmann--Pl\"ucker ideal is generated by the polynomial 
 $m_{12}m_{34} - m_{13}m_{24} + m_{14}m_{23}$. To simplify notation, we set
 $x := m_{12}m_{34}$, $y := m_{13}m_{24}$, and $z := m_{14}m_{23}$ 
 and consider the ideal $I:=(x-y+z)\subseteq \CC[x,y,z]$.
 Since $\sigma$ and $-\sigma$ yield the same ideal, it is sufficient to
 use only $4$ out of the $2^3=8$ possible sign patterns.
 \begin{verbatim}
> ring R = 0, (x,y,z), dp; ideal I = x-y+z; 
> map phi2a = R,x2,y2,z2; ideal J2a = preimage(R,phi2a,I); J2a;
J2a[1]=x2-2xy+y2-2xz-2yz+z2
> map phi2b = R,-x2,-y2, z2; ideal J2b = preimage(R,phi2b,I); 
> map phi2c = R,-x2, y2,-z2; ideal J2c = preimage(R,phi2c,I);
> map phi2d = R, x2,-y2,-z2; ideal J2d = preimage(R,phi2d,I);  
> ideal K = std( std(I+J2a)*std(I+J2b)*std(I+J2c)*std(I+J2d) );
 \end{verbatim}

The ideal $K$ is rather complicated, it is generated by
\begin{align*}
 K &= \bigl(x^{4}-4x^{3}y+6x^{2}y^{2}-4xy^{3}+y^{4}+4x^{3}z-12x^{2}yz+12xy^{2}z-4y^{3}z
 \\ & \qquad \quad
 +6x^{2}z^{2} -12xyz^{2}+6y^{2}z^{2}+4xz^{3}-4yz^{3}+z^{4},
 \displaybreak[2]
 \\
 & \qquad
x^{3}z^{2}-3x^{2}yz^{2}+3xy^{2}z^{2}-y^{3}z^{2}+3x^{2}z^{3}-6xyz^{3}+3y^{2}z^{3}+3xz^{4}-3yz^{4}+z^{5},
 \displaybreak[1]\\
 & \qquad
x^{3}yz-3x^{2}y^{2}z+3xy^{3}z-y^{4}z+3x^{2}yz^{2}-6xy^{2}z^{2}+3y^{3}z^{2}+3xyz^{3}-3y^{2}z^{3}+yz^{4},
 \displaybreak[1]\\
 & \qquad
x^{3}y^{2}-3x^{2}y^{3}+3xy^{4}-y^{5}+3x^{2}y^{2}z-6xy^{3}z+3y^{4}z+3xy^{2}z^{2}-3y^{3}z^{2}+y^{2}z^{3},
 \displaybreak[1]\\
 & \qquad
x^{2}z^{4}-2xyz^{4}+y^{2}z^{4}+2xz^{5}-2yz^{5}+z^{6},
 \displaybreak[1]\\
 & \qquad
x^{2}yz^{3}-2xy^{2}z^{3}+y^{3}z^{3}+2xyz^{4}-2y^{2}z^{4}+yz^{5},
 \displaybreak[1]\\
 & \qquad
x^{2}y^{2}z^{2}-2xy^{3}z^{2}+y^{4}z^{2}-x^{2}yz^{3}+4xy^{2}z^{3}-3y^{3}z^{3}-2xyz^{4}+3y^{2}z^{4}-yz^{5},
 \displaybreak[1]\\
 & \qquad
x^{2}y^{3}z-2xy^{4}z+y^{5}z+2xy^{3}z^{2}-2y^{4}z^{2}-x^{2}yz^{3}+2xy^{2}z^{3}-2xyz^{4}+2y^{2}z^{4}-yz^{5},
 \displaybreak[1]\\
 & \qquad
x^{2}y^{4}-2xy^{5}+y^{6}+2xy^{4}z-2y^{5}z+y^{4}z^{2},
 \displaybreak[1]\\
 & \qquad
xy^{2}z^{4}-y^{3}z^{4}-xyz^{5}+2y^{2}z^{5}-yz^{6},
 \displaybreak[1]\\
 & \qquad
xy^{4}z^{2}-y^{5}z^{2}-2xy^{3}z^{3}+3y^{4}z^{3}+xy^{2}z^{4}-3y^{3}z^{4}+y^{2}z^{5},
 \displaybreak[1]\\
 & \qquad
xy^{5}z-y^{6}z+y^{5}z^{2}-2xy^{3}z^{3}+2y^{4}z^{3}-2y^{3}z^{4}+xyz^{5}-y^{2}z^{5}+yz^{6},
 \displaybreak[1]\\
 & \qquad
y^{6}z^{2}-3y^{5}z^{3}+4y^{4}z^{4}-3y^{3}z^{5}+y^{2}z^{6} \bigr).
%
%
\end{align*}
\begin{align}
\sqrt{K} &= ( x-y+z,\, y^{4}z-2y^{3}z^{2}+2y^{2}z^{3}-yz^{4}) \displaybreak[2]\\
V(K) &= \Bigg\{ (1,0,-1),\, (1,1,0),\, (0,1,1),\, \left(\frac 12 + \frac{\sqrt{3}}{2}i,\, 1,\, \frac 12 - \frac{\sqrt{3}}{2}i \right),
  \displaybreak[1]
  \notag\\ & \qquad \quad
     \left(\frac 12 - \frac{\sqrt{3}}{2}i,\, 1,\, \frac 12 + \frac{\sqrt{3}}{2}i \right)
        \Bigg\} \, \CC.
 \notag
\end{align}
 As described above, $V(K)\subseteq \CC^3$ is the image of 
 the subvariety of $ \CC^6$ defined in  \eqref{eq:VarietySquareable}
 under the map 
\begin{equation}
 (m_{12},m_{13},m_{14},m_{23},m_{24},m_{34}) \mapsto (m_{12}m_{34}, m_{13}m_{24}, m_{14}m_{23}).
\end{equation} 
The variety $V(K)$ can be found using Singular by distinguishing the two cases $y=0$ and $y \neq 0$. For example, to obtain the 
four solutions that satisfy $y=1$,  we used the commands
\begin{verbatim}
> ideal T1 = std(K), y-1; solve(T1);
\end{verbatim} 
Let us finish by constructing a matrix $X$ whose Pl\"ucker coordinates are projected to
$p := \left(\frac 12 + \frac{\sqrt{3}}{2}i,\, 1,\, \frac 12 - \frac{\sqrt{3}}{2}i \right)$.
Let $\sigma=(-1,1,-1)$.
We note that $p$ is contained in $f_{\sigma,2}( V(x-y+z) ) = V( x-y+z, x^2 + xy + z^2 )$
and $f_{\sigma,2}(p) = ( \frac 12 - \frac{\sqrt 3}{2}i, 1, \frac{1}{2} + \frac{\sqrt{3}}{2} i)$.
We choose the following Pl\"ucker coordinates in its preimage:
$m_{12}=m_{13}=m_{14}=m_{24}=1$, $m_{23} = \frac{1}{2} - \frac{\sqrt{3}}{2}i$, and $m_{34}=\frac{1}{2} + \frac{\sqrt{3}}{2}i$.
We may assume that the first two columns of $X$ form a diagonal matrix and that the top-left entry is equal to $1$.
This leads to the following matrix $X$ and the matrix $X_2$ that corresponds to the coordinate-wise square of $p$:

\begin{equation}
  \label{equation:ComplexU24squaring}
  X = \begin{pmatrix}
       1 & 0 & -\frac 12 + \frac{\sqrt 3}{2} i & -1 \\
       0 & 1 &        1                      &  1
      \end{pmatrix}
 \qquad
  X_2 = \begin{pmatrix}
       1 & 0 &  -\frac 12 - \frac{\sqrt 3}{2}i & -1 \\
       0 & 1 &        1                      &  1
      \end{pmatrix}.
\end{equation}

\end{Example}

\section{Background on matroids and arithmetic matroids}
\label{Section:Background}

 In this section we will introduce the main objects of this paper,
 matroids and arithmetic matroids.

\subsection{Matroids}

A \emph{matroid} is a pair $(E,\rank)$, where $E$ denotes a finite set
and the rank function $\rank : 2^E \to \Z_{\ge 0}$ satisfies the following axioms:
\begin{itemize}
 \item $0\le \rank(A) \le \abs{A}$ for all $A\subseteq E$, 
 \item $A\subseteq B \subseteq E$ implies $\rank(A) \le \rank(B)$, and
 \item  $\rank(A\cup B) + \rank(A\cap B) \le \rank(A) + \rank(B)$  for all $A,B\subseteq E$. 
\end{itemize}
 A standard reference on matroid theory is Oxley's book \cite{MatroidTheory-Oxley}.

 Let $\K$ be a field and let  $E$ be a finite set, \eg $E=[N]:=\{1,\ldots, N\}$.
 A finite list of vectors $X=(x_e)_{e\in E}$ in $\K^d$ 
 defines a matroid in a canonical way:
 the ground set is $E$  and the rank function 
 is the rank function from linear algebra.
 A matroid that  can be obtained 
 in such a way is called representable over $\K$. Then the list $X$ is called a representation of the matroid.
 Of course,  a list of vectors $X=(x_e)_{e\in E}$ is essentially the same as a matrix
   $X\in \K^{d \times \abs{E}}$ whose columns are indexed by $E$.

The \emph{uniform matroid} $U_{r, N}$ of rank $r$ on $N$ elements is the matroid on the ground set $[N] $,  whose rank function
is given by  $\rank(A)=\min( \abs{A}, r)$ for all $A\subseteq [N]$.

 A graph $\Gcal=(V,E)$ that
 may contain loops and multiple edges defines a \emph{graphic matroid}. Its ground set is the set $E$ of edges of the graph and the rank of a set of edges is defined as the cardinality
of a maximal subset that does not contain a cycle.  Graphic matroids can be represented over any field by an oriented vertex-edge incidence matrix of the graph.

\subsection{Arithmetic matroids}
\label{Subsection:AriMatroids}
 Arithmetic matroids capture many combinatorial and topological properties of toric arrangements 
 in a similar way as matroids carry information about the corresponding hyperplane arrangements.
\begin{Definition}[D'Adderio--Moci, Br\"and\'en--Moci
\cite{branden-moci-2014,moci-adderio-2013}]
 An arithmetic matroid is a triple $(E, \rank, m)$, where
 $(E, \rank)$ is a matroid and $m : 2^E \to \Z_{\ge 1}$ denotes the
 \emph{multiplicity function}
 that satisfies the following axioms:
 \begin{itemize}
   \item[(P)]  Let $R\subseteq S\subseteq E$. The set $[R,S]:=\{A: R\subseteq A\subseteq S\}$ is called a \emph{molecule} 
 if $S$ can be written as the disjoint union $S=R\cup F_{RS}\cup T_{RS}$ and for each $A\in [R,S]$,
 $\rank(A)= \rank(R) + \abs{A \cap F_{RS}}$ holds.
 
 For each molecule   $[R,S]$, the following inequality holds:
 \begin{equation}
  \label{eq:Paxiom}
  \rho(R,S) := (-1)^{\abs{T_{RS}}} \sum_{A\in [R,S]} (-1)^{\abs{S} - \abs{A} } m(A) \ge 0.
 \end{equation}
  \item[(A1)]  For all $A \subseteq E$ and $e \in E$:
 if $\rank(A \cup \{e\}) = \rank(A)$, then $m(A \cup \{e\}) \big| m(A)$.
 Otherwise, $m(A) \big| m(A \cup \{e\})$.
  \item[(A2)]  If $[R, S]$ is a molecule and $S=R\cup F_{RS}\cup T_{RS}$, then 
  \begin{equation}
     m(R)m(S) = m(R \cup F_{RS} ) m(R \cup T_{RS} ).
  \end{equation}
\end{itemize}

\end{Definition}

 The prototypical example of an arithmetic matroid 
 is defined by a list of vectors $X = (x_e)_{e\in E} \subseteq\Z^d$.
 In this case, for a subset  $S\subseteq E$ of cardinality $d$ that defines a basis, we have $m(S) = \abs{\det(S)}$
 and in general $m(S) := \abs{ (\left\langle  S\right\rangle_\R \cap \Z^d) / \left\langle S\right\rangle}$. 
 Here, $\langle S \rangle\subseteq \Z^d$ denotes the subgroup generated  by $\{x_e: e\in S\}$ and $\langle S \rangle_\R\subseteq \R^d$
 denotes the subspace spanned by the same set.
 
 Below, we will see that representations of contractions of arithmetic matroids are contained in a quotient of the ambient group.
 As quotients of $\Z^d$ are in general not free groups, 
 we will work in the slightly more general setting of finitely generated abelian groups.
 By the fundamental theorem of finitely generated abelian groups, every finitely generated abelian group $G$ is isomorphic
 to $\Z^d \oplus \Z_{q_1} \oplus \ldots \Z_{q_n}$ for  suitable non-negative integers $d, n, q_1,\ldots, q_n$.
 There is no canonical isomorphism $G \cong \Z^d \oplus \Z_{q_1} \oplus \ldots \Z_{q_n}$.
  However, $G$ has a uniquely determined subgroup $G_t \cong \Z_{q_1} \oplus \ldots \Z_{q_n}$ consisting of all the torsion elements.
 There is a free group
 $\latproj{G} := G/ G_t \cong \Z^d$.
 For $X\subseteq G$, we will write $\latproj{X}$ to denote the image of $X$ in $\latproj{G}$.

Note that a finite list of vectors $X=(x_e)_{e\in E} \subseteq \Z^d \oplus \Z_{q_1} \oplus \ldots \Z_{q_n}$ can be 
identified with an $(d + n)\times \abs{E}$ matrix, where each column corresponds to one of the vectors. 
The first $d$ rows of the matrix consists of integers and the entries of the remaining rows are contained in certain cyclic groups.

\begin{Definition}
 Let $\Acal = (E, \rank, m)$ be an arithmetic matroid.
 Let $G$ be a finitely generated abelian group and let $X = (x_e)_{e\in E}$ be a list of elements of $G$.
 For $A\subseteq E$, let  $G_A$ denote the maximal subgroup of $G$ \st $ \abs{G_A / \left\langle A \right\rangle}$ is finite.
 Again, $\left\langle A \right\rangle \subseteq G$ denotes the subgroup generated by $\{x_e:e\in A\}$.

  $ X \subseteq G $ is called a \emph{representation} of $\Acal$ if the matrix $\bar X\subseteq \bar G$ represents the matroid $(E,\rank)$ 
  and 
  $m(A) = \abs{G_A / \left\langle A \right\rangle}$ for all $A\subseteq E$.
  The arithmetic matroid $\Acal$ is called \emph{representable} if it has a representation $X$.
  Given a  list $X = (x_e)_{e\in E}$  of elements of a finitely generated abelian group $G$, we will
  write $\Acal(X)$ to denote the arithmetic matroid $(E,\rank_X,m_X)$ represented by $X$.
\end{Definition}

 Unfortunately, allowing the representation to be contained in an arbitrary finitely generated abelian group
 makes certain statements and proofs more complicated.
 A reader with little  knowledge about arithmetic matroid should first consider the case  of representations being
 contained in $\Z^d$. This setting captures the most interesting parts of the theory.

 An arithmetic matroid $\Acal = (E, \rank, m)$ is called \emph{torsion-free} if 
 $m(\emptyset)=1$. If $\Acal$ is representable, then it can be represented by a list of vectors in a lattice, \ie a finitely generated abelian group that is torsion-free.

\subsection{Hermite normal form}
\label{Subsection:HNF}
 We say that a matrix $X\in \Z^{d\times N}$ of rank $r$ ($r \le d\le N$) is in \emph{Hermite normal form} 
 if for all $i \in [r]$, $0 \le x_{ij} < x_{jj}$ for $i<j$ and $x_{ij}=0$ for $i>j$, \ie
 the first $r$ columns of $X$ form an upper triangular matrix and the diagonal elements are strictly
 bigger than the other elements in the same column.
 It is not completely trivial, but well-known,  that any matrix $X \in \Z^{d\times N}$  %
 can be brought into Hermite normal form
 by permuting the columns and multiplying it from the left with a unimodular matrix $U\in \GL(d, \Z)$  
 (\cite[Theorem~4.1 and Corollary~4.3b]{schrijver-TheoryLinearIntegerProgramming}).
 Since such a multiplication does not change the arithmetic matroid represented by the matrix, 
 we will be able to assume that a representation $X$ of a torsion-free arithmetic matroid $\Acal$ is in Hermite normal form.

 \subsection{Restriction and contraction}
 \label{Subsection:RestrictionContraction}
 
 It is possible to extend the matroid operations restriction and contraction to arithmetic matroids
 \cite{moci-adderio-2013}.
 Let $\Acal = (E, \rank, m)$ be an arithmetic matroid.
 Let $A\subseteq E$.  
 The \emph{restriction} $\Acal|_A$ is the arithmetic matroid 
 $\Acal|_A = (A,\rank|_A, m|_A)$, where  $\rank|_A$ and $m|_A$ denote the restrictions of $\rank$ and $m$ to $A$.
 The \emph{contraction} $\Acal / A$ is the
 arithmetic matroid   $(E\setminus A,\rank_{/A}, m_{/A})$, where
 $\rank_{/A}(B):=\rank(B\cup A) - \rank(A)$ and
  $ m_{/A}(B) := m(B\cup A)$ for $B\subseteq E\setminus A$.
 Recall that for a matroid $\Mcal=(E,\rank)$ and $A\subseteq E$, restriction and contraction are defined as 
 $\Mcal|_A:=(A, \rank|_A)$ and $\Mcal/A := (E\setminus A, \rank_{/A})$.

  If an arithmetic matroid $(E, \rank , m)$ is represented by a list $X = (x_e)_{e\in E}$ of elements of $G$,
  it is easy to check that the restriction corresponds to the arithmetic matroid represented by the list  
  $X|_A := (x_e)_{e\in A}$. The contraction $\Acal/A$ is represented by the sublist 
  $X/A := (\bar x_e)_{e\in E\setminus A}$ consisting of the images of $ (x_e)_{e\in E\setminus A} $ under the canonical projection
  $ G  \twoheadrightarrow G / \langle A \rangle $ (cf.~\cite[Example~4.4]{moci-adderio-2013}).

\begin{Example}
 Let us consider the matrix
 $ X = \begin{pmatrix} 1 & 1 & 2 \\ 0 & 3 & 4 \end{pmatrix}$.
 What happens if we contract the last column?
 Since we are mainly interested in the underlying arithmetic matroid, 
 we will first apply a unimodular transformation \st the last column only contains one 
 non-zero entry (cf.~Subsection~\ref{Subsection:HNF}) and consider the resulting matrix $X'$. 
 Then we obtain
 $ X' = \begin{pmatrix} 1 & 1 & 2 \\ -2 & 1 & 0 \end{pmatrix} =(x_1,x_2,x_3)$. 
  Hence $ X'/x_3 = \begin{pmatrix} \bar 1 & \bar 1  \\ -2 & 1  \end{pmatrix} \subseteq \Z_2 \oplus \Z$ and
 $\latproj{X} = (-2,1)\subseteq \Z$ is the image of $ X'/x_3$ under the projection to the free group obtained from $\Z^2/\left\langle x_3\right\rangle$ by taking the quotient by
 its torsion subgroup.
\end{Example}
 \subsection{Regular matroids}
 \label{Subsection:RegularMatroids}
 
  Recall that a matroid is \emph{regular} if it can be represented over every field,
 or equivalently, if it can be represented by a totally unimodular matrix.
 There are several other  characterisations of regular matroids.
 \begin{Theorem}[Seymour \cite{seymour-1980}, see also {\cite[Theorem~13.2.4]{MatroidTheory-Oxley}}]
 \label{Theorem:SeymourRegular}

  Every regular matroid $M$ can be constructed by means of direct sums, $2$-sums, and $3$-sums starting with matroids each of which is isomorphic to a minor 
  of $M$, and each of which is either graphic, cographic, or isomorphic to $R_{10}$.
 \end{Theorem}
 A direct sum  is the matroidal analogue of glueing two graphs in a vertex, a $2$-sum
 corresponds to glueing two graphs in an edge and 
 the $3$-sum is corresponds to glueing two graphs in a $K_3$.
 $R_{10}$ denotes a certain rank $5$ matroid on $10$ elements.

The following   result and its corollary are used in the proofs of 
Theorem~\ref{Theorem:NonRegularMatricesNewNew}, 
Theorem~\ref{Theorem:NonRegularNonRepresentable}, and
Proposition~\ref{Proposition:IdealRegularGrassmannian}.
\begin{Theorem}[{\cite[Theorem~6.6.4]{MatroidTheory-Oxley}}]
  \label{Theorem:RegularExcludedMinors}
   A matroid is regular if and only if it has no minor isomorphic to $U_{2,4}$, the Fano matroid or its dual.
\end{Theorem}

 The Fano matroid is the rank $3$ matroid that is represented by the list of all $7$ non-zero vectors in $(\F_2)^3$.
 The Fano matroid and its dual are representable over a field $\K$ if and only if this field has characteristic $2$  \cite[Proposition~6.4.8]{MatroidTheory-Oxley}.
 Hence we immediately obtain the following corollary.

\begin{Corollary}
  \label{Corollary:RegularExcludedMinors}
  Let $M$ be a matroid that is representable over a field $\K$ whose characteristic is not $2$.
  Then $M$ is regular if and only if it has no minor isomorphic to $U_{2,4}$.  
\end{Corollary}

\section{Proofs of the main results on non-regular matroids}
\label{Section:Proofs}

 In this  section  we will prove 
 Theorem~\ref{Theorem:NonRegularMatricesNewNew}
 and
 Theorem~\ref{Theorem:NonRegularNonRepresentable}.
 The main ingredient of the proofs is the Grassmann--Pl\"ucker relation (cf.~Theorem~\ref{Theorem:GPrelations}) between the minors of a $(2\times 4)$-matrix.
 Recall that for a matrix $X\in \K^{2\times 4}$  
 \begin{equation}
   \label{eq:SimpleGP}
   \Delta_{12} \Delta_{34}  - \Delta_{13} \Delta_{24} + \Delta_{14}\Delta_{23} = 0 
 \end{equation}
 holds, where
 $\Delta_{ij}$ denotes the determinant of the columns $i$ and $j$ of $X$. 
 If columns $i$ and $j$ of $X\in \Z^{2\times 4}$ are linearly independent, 
 then the multiplicity function of the arithmetic 
 matroid $\Acal(X)$ satisfies  $m_X(\{i,j\}) = \abs{\Delta_{ij}}$.
 We can immediately deduce from the Grassmann--Pl\"ucker relation
 that there is no representable arithmetic matroid with underlying matroid $U_{2,4}$ whose 
 bases all have multiplicity one.

 Recall that for a matrix
 $X\in \K^{d\times N}$ ($ d\le N$) and  $I\in \binom{[N]}{d}$, $\Delta_I=\Delta_I(X)$ denotes the minor corresponding to the columns indexed by $I$.
 
 \begin{Lemma}
 \label{Lemma:U24}
  \begin{enumerate}[(i)]
   \item  Let $\Acal=(E, \rank, m)$ be a representable arithmetic matroid with underlying matroid $U_{2,4}$.
  Then for any non-negative integer  $k\neq 1$, $\Acal^k:=(E,\rank, m^k)$ %
  is not representable.
   \item Let $X \in \R^{2\times 4}$ be a matrix that represents the matroid $U_{2,4}$, 
  \ie the columns of $X$ are pairwise linearly independent.
    Let $k$ be a non-negative real number that satisfies $k\neq 1$.
  Then there is no $X_k \in \R^{2\times 4}$ 
  \st for each maximal minor $\Delta_I(X)$, the corresponding minor $\Delta_{I}(X_k)$
   satisfies $\abs{\Delta_I(X)}^k = \abs{\Delta_I(X_k)}$.   
   
   If $k$ is an integer, then the same statement holds over any ordered field $\K$.
  \end{enumerate}

 \end{Lemma}
 \begin{proof}
   We will start with the proof of the first statement.
   We assume that $E=[4]$.
   Let $ X \subseteq G $ denote a representation of $\Acal$ in a finitely generated abelian group $G$ and let $\latproj{X}$ denote its projection to the free group $G/G_t$.
   We can identify $\latproj{X}$ with a matrix in $\Z^{2\times 4}$.
   Let $\Delta_{ij}$ denote the determinant  of the columns $i$ and $j$ of $\latproj{X}$. 
   It is easy too see (see also Lemma~\ref{Lemma:TorsionIndependent})  that $m_{ij}:=m_\Acal(\{i,j\}) = \abs{G_t} \cdot \abs{\Delta_{ij}} $.
   From the  
   Grassmann--Pl\"ucker relation \eqref{eq:SimpleGP} it follows that
   \begin{equation}
    \label{equation:GP}
    \underbrace{\sigma_1 m_{12}m_{34}}_{\alpha} + \underbrace{\sigma_2 m_{13}m_{24}}_{\beta} + \underbrace{\sigma_3 m_{14}m_{23}}_{\gamma} = 0
   \end{equation}
   for suitable  $\sigma_1, \sigma_2, \sigma_3\in \{ -1, 1 \}$. %
   By the pigeonhole principle, at least two out of $\alpha$, $\beta$, and $\gamma$ have the same sign.
   Hence, without loss of generality, we may assume that $\beta \gamma> 0$.
   Furthermore, we can assume $\alpha>0$. Otherwise, we just multiply \eqref{equation:GP} by $-1$, which leaves the product $\beta\gamma$ unchanged.
   Hence   $\alpha = - (\beta + \gamma) = \abs{\beta+\gamma}$.
      This implies
   \begin{equation}
   \label{eq:alphabetagamma}
     \alpha^k =   \abs{\beta + \gamma}^k = (\abs{\beta} + \abs{\gamma})^k.
   \end{equation}

 Suppose there is a   list of vectors $X_k$ contained in a finitely generated abelian group $G'$ that represents $\Acal^k$.
 Now we proceed as above:
 let $\latproj{X'}$ denote the projection of $X'$ to the free group $G'/G'_t$.
  We can identify $\latproj{X'}$ with a matrix in $\Z^{2\times 4}$.
  Let $\Delta'_{ij}$ denote the determinant  of the columns $i$ and $j$ of $\latproj{X'}$. 
   Then  $m_{ij}^k = m_{\Acal^k}(\{i,j\}) = \abs{G'_t} \cdot \abs{\Delta'_{ij}} $.
 This implies for $k\ge 2$:
   \begin{equation}
   \begin{split}
  0 &=  \underbrace{m_{12}^km_{34}^k}_{\alpha^k} \pm \underbrace{  m_{13}^k m_{24}^k}_{\abs{\beta}^k} \pm \underbrace{  m_{14}^k m_{23}^k}_{\abs{\gamma}^k} 
  = \alpha^k \pm \abs{\beta}^k \pm \abs{\gamma}^k 
  \\ &
  \stackrel{\eqref{eq:alphabetagamma}}{=} (\abs{\beta} + \abs{\gamma})^k \pm \abs{\beta}^k \pm  \abs{\gamma}^k
  \ge 
   (\abs{\beta} + \abs{\gamma})^k - \abs{\beta}^k -  \abs{\gamma}^k
\stackrel{k>1}{>} 0. %
  \end{split}
   \label{eq:kGP}
   \end{equation}
  This is a contradiction.
  The last step holds because the function $f(x)=x^k$ is strictly convex for $x \ge 0$ and $k>1$.
  It is known that for functions $f : [0,\infty)\to \R$ that satisfy $f(0)=0$, convexity implies superadditivity, 
  \ie 
  $f(\abs{\beta} + \abs{\gamma}) >  f(\abs{\beta}) + f(\abs{\gamma})$.

   For $k=0$, we obtain  $0 =  \alpha^0 \pm \abs{\beta}^0 \pm \abs{\gamma}^0  = 1 \pm 1 \pm 1$. 
   But $1 \pm 1 \pm 1$ is odd, which is a contradiction. This finishes the proof of the first statement.

 \medskip
   The proof of the second statement over $\R$ is very similar.
   Let $\Delta_{ij}$ denote the determinant  of the columns $i$ and $j$ of $X$. 
   We define $m_{ij}:=  \abs{\Delta_{ij}} \neq 0$.
   For $k>1$ and $k=0$, we can proceed as above.
   The case $0< k<1$ is slightly more complicated.
   We start as in \eqref{eq:kGP} and obtain
 \begin{equation}
  0 =  \underbrace{m_{12}^km_{34}^k}_{\alpha^k} \pm \underbrace{  m_{13}^k m_{24}^k}_{\abs{\beta}^k} \pm \underbrace{  m_{14}^k m_{23}^k}_{\abs{\gamma}^k}
  \stackrel{\eqref{eq:alphabetagamma}}{=} (\abs{\beta} + \abs{\gamma})^k \pm \abs{\beta}^k \pm  \abs{\gamma}^k.
 \end{equation}
   Now we distinguish three cases:
  if both signs are negative, then  $(\abs{\beta} + \abs{\gamma})^k - \abs{\beta}^k -  \abs{\gamma}^k < 0$ since
   $0 < k < 1$:
    the function $f(x)=x^k$ is strictly concave for $x \ge 0$ and $0<k<1$.
  It is known that for functions $f : [0,\infty)\to \R$ that satisfy $f(0)=0$, concavity implies subadditivity, 
  \ie 
  $f(\abs{\beta} + \abs{\gamma}) <  f(\abs{\beta}) + f(\abs{\gamma})$.

   If exactly one sign is positive (\WLOG\ the first), we obtain
   \begin{equation}
      \underbrace{(\abs{\beta} + \abs{\gamma})^k  -  \abs{\gamma}^k}_{> 0} + \abs{\beta}^k  > \abs{\beta}^k >  0.
   \end{equation}
   The case that both signs are positive can be reduced to the second, since $(\abs{\beta} + \abs{\gamma})^k + \abs{\beta}^k +  \abs{\gamma}^k >
   (\abs{\beta} + \abs{\gamma})^k + \abs{\beta}^k -  \abs{\gamma}^k$.
   In all three cases we have reached a contradiction.

   The proof of the second statement over an ordered field $\K$ is also similar to the proof of the first statement.
   For $k=0$, it is clear. For $k\ge 2$, we proceed as above.
   In the last step, the convexity argument cannot be used to prove that
   $(\abs{\beta} + \abs{\gamma})^k - \abs{\beta}^k -  \abs{\gamma}^k > 0$.
   Instead, we can use the binomial theorem.
 \end{proof}

   \begin{proof}[Proof of Theorem~\ref{Theorem:NonRegularMatricesNewNew}]
     We will consider the two cases $\K=\R$ and $\K$ an ordered field at the same time.
     Let $E$ be a set of cardinality $N$ and let $X=(x_e)_{e\in E}\subseteq \K^d$ be a list of vectors that spans $\K^d$. 
     Suppose there is a list of vectors $X_k=(x_e')_{e\in E}$ \st $ \abs{ \Delta_I(X) }^k = \abs{ \Delta_I(X_k) }$
     for all $ I \in \binom{ E }{ d }$.
     By Corollary~\ref{Corollary:RegularExcludedMinors} and since ordered fields always have characteristic $0$, both $X$ and $X_k$ must have a $U_{2,4}$ minor.
     In fact, since  $X$ and $X_k$ define the same labelled matroid, there must be disjoint subsets $I,J\subseteq E$ \st
      $A := (X/J)|_I$ and 
     $A_k := (X_k/J)|_I$ represent $U_{2,4}$.
     These two matrices will allow us to reach a contradiction using  Lemma~\ref{Lemma:U24}.

     Note that $X|_J$ spans a subspace of dimension at most $d-2$.
     Let $J_0 \subseteq J$ be a  basis for this subspace.
     Furthermore, let $K_0\subseteq E \setminus (I \cup J)$ be   minimal   \st $J_0 \cup I \cup K_0$ has full rank.
     After a change of basis, encoded by a matrix $T\in \GL(d,\K)$, we can assume that $X$ has the following shape:
\begin{equation}
 TX =
\begin{blockarray}{rrrrr} %
 I & J_0 & J \setminus J_0 & K_0 & \\
\begin{block}{(r|r|c|r|r)}
 A  &  0  & *   &  0 & * & \\ \BAhline
 *  &  B  & *   &  0 & * \\ \BAhline
 *  &  0  & 0   &  C & *\\ 
\end{block}
\end{blockarray}\,,
\end{equation}
 where $A\in \K^{2\times 4}$, $B\in \K^{\abs{J_0} \times \abs{J_0}}$, $C \in \K^{\abs{K_0} \times \abs{K_0}}$, $0$ denotes zero matrices and $*$ denotes arbitrary
 matrices of suitable dimensions.
     Let $i,j\in I$, $i\neq j$. Since $A$ represents $U_{2,4}$,
     $\{i,j\}\cup J_0 \cup K_0$ is a basis for the matroid represented by $X$ and
     $  \Delta_{\{i,j\}\cup J_0 \cup K_0}(X) \det(T)  =   \Delta_{ij}(A) \det(B) \det(C)$.

 As both matrices represent the same matroid, $\{i,j\}\cup J_0 \cup K_0$ is also a basis 
 for $X_k$
 and $(X_k)|_{J}$ spans a space of the same dimension as $X|_J$.
  Hence after a change of basis, encoded by a matrix $T_k\in \GL(d,\K)$, we can assume that $X_k$ has the following shape:
\begin{equation}
 T_k X_k =
\begin{blockarray}{rrrrr} %
 I & J_0 & J \setminus J_0 & K_0 & \\
\begin{block}{(c|c|c|r|r)}
 A_k  &  0  & *   &  0 & * & \\ \BAhline
 *  &  B_k  & *   &  0 & * \\ \BAhline
 *  &  0  & 0   &  C_k & *\\ 
\end{block}
\end{blockarray}\,.
\end{equation}
 Hence $\Delta_{\{i,j\}\cup J_0 \cup K_0}(X_k) \det(T_k) =  \Delta_{ij}(A_k) \det(B_k) \det(C_k) $.
 By assumption, we have  
 \begin{equation*}
 \abs{\Delta_{\{i,j\}\cup J_0 \cup K_0}(X_k)} = \abs{\Delta_{\{i,j\}\cup J_0 \cup K_0}(X)}^k.
 \end{equation*}
 We can deduce that
 there is  $\kappa\in \K^*$ \st
 $ \kappa \abs{\Delta_{ij}(A_k)} =  \abs{\Delta_{ij}(A)}^k$ 
 for all $i,j\in I$, $i\neq j$.
 Now let $A_k' \in \K^{2\times 4}$ be the matrix that is obtained from $A_k$ by scaling the first   row by the factor $\kappa$.
 Then  $ \abs{\Delta_{ij}(A_k')} = \kappa \abs{\Delta_{ij}(A_k)} =  \abs{\Delta_{ij}(A)}^k$ for all $i,j\in I$, $i\neq j$.
 The existence of $A_k'$ and $A$ is  
 a contradiction to Lemma~\ref{Lemma:U24}.
\end{proof}
 
 \begin{proof}[Proof of Theorem~\ref{Theorem:NonRegularNonRepresentable}]
  By assumption, there is a matrix $X$ that represents $\Acal = (E, \rank, m)$.
  Suppose there is a matrix $X_k$ that represents $\Acal^k = (E, \rank, m^k)$.
  Both matrices are indexed by $E$.
     By Corollary~\ref{Corollary:RegularExcludedMinors}, the matroid $(E,\rank)$ must have a $U_{2,4}$ minor.
  This means that there are disjoint subsets $I,J\subseteq E$ \st
  $ A := (X/J)|_I  $ has the underlying matroid $U_{2,4}$. %
  Let $A_k  :=  (X_k/J)|_I $.
  Then $A_k$ represents the matroid $\Acal(A)^k$: both have the same underlying matroid
  and for $S\subseteq I$, $m_{A_k}(S) = m_{X_k}(S\cup J) = (m_{X}(S\cup J))^k = (m_{A}(S))^k$.
  By Lemma~\ref{Lemma:U24}, this is a contradiction.
 \end{proof}

 \section{Proof of Proposition~\ref{Proposition:NonRegularMatricesNewNewBestPossible}}
 \label{Section:SupplementaryResults}
  In this section we will prove
  Proposition~\ref{Proposition:NonRegularMatricesNewNewBestPossible}. 
  The proof is rather elementary and 
  independent of most of the rest of the paper.
  \begin{proof}[Proof of Proposition~\ref{Proposition:NonRegularMatricesNewNewBestPossible}]
   Let $\K$ be a field and $a\in \K$. We define the two matrices
    \begin{equation}
    \label{equation:CounterexampleMatrices}
    X(a) := \begin{pmatrix}
          1 & 0 & 1 & 1 \\ 0 & 1 & 1 & a 
        \end{pmatrix}
      \quad\text{and}\quad 
    X_k(a) := \begin{pmatrix}
          1 & 0 & 1 & 1 \\ 0 & 1 & 1 & a^k 
        \end{pmatrix}.
   \end{equation}
  These two matrices will serve as counterexamples for suitable choices of $a$.
  Since we want that both  $X(a)$ and $X_k(a)$ represent $U_{2,4}$, we require that $a$ and $a^k$ are not contained in $\{0,1\}$.
  Then for $X(a)$, ${\Delta_{12}} = {\Delta_{13}} = {\Delta_{23}} = {\Delta_{24}} = \pm 1$, ${\Delta_{14}}=a$, and ${\Delta_{34}}={a-1}$.
  The two matrices provide a counterexample
   if $ \Delta_{34}(X(a))^k = \pm \Delta_{34}(X_k(a))$ holds. For the other minors, the condition is always satisfied.
  Hence we need to check whether the following equation has a solution:
  \begin{equation} 
  \label{equation:Counterexample}
    (a-1)^k = \pm (a^k - 1).
  \end{equation}
  
    Let us   consider the case $\characteristic(\K)=2$ first.
    Over $\F_2$, Theorem~\ref{Theorem:NonRegularMatricesNewNew} is trivially false.
    If $\K\neq \F_2$, the set $\K\setminus\{0,1\}$ is non-empty 
    and for $k=2$, \eqref{equation:Counterexample} is satisfied for any $a\in \K$.

    Now let us consider the case  $\characteristic(\K)= p \ge 3$.
    Then $a=2 \not\in \{0,1\}$. 
    Let $k = \lambda (p-1) + 1$ for some integer $\lambda\ge 1$.
    By Fermat's little theorem  $a^{p-1}=1$. Hence we obtain  $a^k=a=2 $ and $(a-1)^k = 1 = a^k-1$.

    If we choose $k=3$ and the negative sign 
    in \eqref{equation:Counterexample}, we obtain the equation
    $2 a^3 - 3 a^2 + 3 a - 2 = 0 $.
    Over $\CC$, it has the solutions $a=1$ and $a= \frac 14( 1 \pm \sqrt{15}i)$. This yields a counterexample.
    A counterexample over $\CC$ for $k=2$ is given in \eqref{equation:ComplexU24squaring}.
    
   More generally, we are able to construct   counterexamples  for algebraically closed fields of characteristic $p\ge 0$
   for $k\ge 3$ if $p\neq 2$ and $p$ does not divide $k$.
    It is sufficient to show  that $f(a) := (a-1)^k + (a^k -1) = 0$ has a solution different from $0$ or $1$.
    This automatically implies that $a^k$ is different from $0$ or $1$.
   Recall that a polynomial $f$ has a double root in $x_0$ if and only if $f$ and $f'$ have $x_0$ as a root.
   The derivative is defined formally here by $(x^l)'=l x^{l-1}$.
   Note that $f'(a) = k(a-1)^{k-1} + k a^{k-1} = 2ka^{k-1} + \ldots$. 
   Since $p$  does not divide $2k$, this is a polynomial of degree $k-1$   and it is clear that for $ a \in \{ 0, 1\}$, it assumes the value $\pm k \neq 0$.
   Hence $0$ and $1$ are simple roots of $f$ and since $ k \ge 3 $ and $ \K $ is algebraically closed, $f$ must have another root different
   from $0$ and $1$.

   Now let us consider the case $\K=\R$ and $k$ a negative real number.
  If we choose the negative sign,  \eqref{equation:Counterexample} is equivalent to
   \begin{equation}
    f(a):= \left(\frac{a}{ (a-1)}\right)^s  - a^s = - 1, 
   \end{equation}
   where $s:=-k > 0$. Note that the function $f(a)$ is continuous in $a$ for $a\neq 1$.
   $\lim_{a \searrow 1} f(a) = + \infty$ and $\lim_{a \to \infty} f(a) = - \infty$.  Hence by the 
   intermediate value theorem, there 
   must be an $a\in (1,\infty)$ \st $f(a) = -1$.
 \end{proof}

  \section{On regular matroids}
  \label{Section:RegularMatroids}
  \subsection{Representations of regular matroids are unique}
  A key ingredient of the proofs of 
  several of our results
  is the fact that representations of regular matroids are unique, up to certain natural transformations.
  So in particular, every matrix that represents a regular matroid can be expressed as a transformation of a
  totally unimodular matrix.
  
  Let $\K$ be a field. 
  Let $X \in \K^{d\times N}$ and let
  $D,P \in \GL(N, \K)$, where $D$ denotes a non-singular diagonal matrix and $P$ a permutation matrix.
  Let $T \in \GL(d,\K)$ and let $\psi : \K \to \K$ be a field automorphism.
  We will also use the letter $\psi$ to 
  denote the map $\K^{d\times N}\to \K^{d\times N}$ that applies 
  the field automorphism $\psi$ to each entry of a matrix. 
  Recall that $\Q$ and $\R$ do not have any non-trivial field automorphisms.  Complex conjugation is a field automorphism of $\CC$ and 
  further wild automorphisms can be constructed using the axiom of choice.
  It is easy to see that $X$ and $ \psi (T\cdot X \cdot D \cdot P)$ represent the same matroid.
  
  Let $M$ be a matroid of rank $d \ge 1$ on $N$ elements. Let $X_1,X_2\in \K^{d\times N}$ denote matrices that both represent $M$.
  We say that $X_1$ and $X_2$ are \emph{equivalent representations} of $M$, if
  $d\in \{1,2\}$, or $d \ge 3$ and there are $D,P,T$, and $\psi$ as above \st
  $X_2 =  \psi (T\cdot X_1 \cdot P \cdot D )$. 
  A rank $0$ matroid is represented by a zero matrix with the appropriate number of columns.
  This case is not relevant for us and we will always assume that the rank of our matroids is at least $1$.

 \begin{Theorem}[\cite{brylawski-lucas-1976},{\cite[Corollary~10.1.4]{MatroidTheory-Oxley}}]
  \label{Theorem:RegularMatroidsUniqueRep}
    Let $M$ be a regular matroid of rank $d\ge 1$ on $N$ elements and let $\K$ be a field. 
    Then  all representations of $M$ by a $(d\times N)$-matrix over $\K$ are equivalent.
 \end{Theorem}

 Recall that a matrix $ A\in \K^{d\times N} $ is \emph{totally unimodular} if all its minors are $0$, $1$, or $-1$.
 Of course, in characteristic $2$, $ -1=1$ holds, but the same definition is used.

\begin{Corollary}
\label{Corollary:BLrepresentationOfMatroid}
  Let $\K$ be a field and let $ X_1, X_2 \in \K^{d\times N}$ be two matrices that represent the same matroid of rank  $d \ge 3$.
  Then there exists  
  a non-singular diagonal matrix $D\in \GL(N,\K)$, a matrix $T \in \GL(d,\K)$, a permutation matrix $P\in \GL(d,\K)$  and a field automorphism $\psi : \K\to \K$
  \st
  $ X_1 = \psi( T X_2 D)$.

\end{Corollary}
\begin{proof}
    Let $M$ be the matroid represented by $X_1$ and $X_2$.
    By Theorem~\ref{Theorem:RegularMatroidsUniqueRep}, $X_1$ and $X_2$ are equivalent representations of $M$.
    Hence there are matrices
    $D,P \in \GL(N, \K)$ ($P$ permutation matrix and $D$ non-singular diagonal matrix), 
    $T \in \GL(d,\K)$ and a field automorphism $\psi : \K \to \K$ 
    \st $X_1 =   \psi (T \cdot X_2 \cdot P \cdot D )$.
\end{proof}

 \begin{Corollary}
 \label{Corollary:BLrepresentationOfRegularMatroid}
  Let $ X \in \K^{d\times N}$ be a matrix that represents a regular matroid $M$ of rank $d \ge 1$.
 Then there exists a totally unimodular matrix $A \in \K^{d\times N}$ that represents $M$,  
  a non-singular diagonal matrix $D\in \GL(N,\K)$,
  and a matrix $T\in \GL(d,\K)$ 
  \st
  $X =  T A D$.
\end{Corollary}
\begin{proof}
 Let us first suppose that  the matroid $M$ has rank at least $3$.
    Then there is a totally unimodular matrix $A' \in \K^{d\times N}$ that   represents $M$. 
    By Corollary~\ref{Corollary:BLrepresentationOfMatroid},
    there exist a diagonal matrix $D'\in \GL(N,\K)$, a matrix $T' \in \GL(d,\K)$ and a field automorphism $\psi : \K \to \K$
    \st      $X = \psi( T' A' P D)  = \psi(T') \psi(A' P) \psi(D')$.
    Since $0$, $1$ and $-1$ are fixed by any field automorphism and $A'P$ is totally unimodular, the matrix $A := A' P = \psi(A' P) $ is also totally unimodular.
    Furthermore, 
    $T:=\psi(T')\in \GL(d,\K)$ and $D:=\psi(D')$ is a non-singular diagonal matrix.
    Hence $X = TAD$ as required.
    
    Now suppose that $ \rank(M) \le 2 $. 
    If $\rank(M)=1$ and $X=(x_1,\ldots, x_l, 0, \ldots, 0)$ for some $x_i\in \K^*$, we have $X= (1) \cdot (1,\ldots, 1,0,\ldots, 0) \cdot \diag(x_1,\ldots, x_l,1,\ldots, 1)$.
    Now we consider the case $\rank(M)=2$. Since $M$ is regular, $X$ may only contain three different types of vectors 
    (up to scaling).
   Hence, after applying a suitable transformation $T$ from the left, we may assume that $X$ has the following shape:
    \begin{equation}
        T X = \left( \begin{array}{cccccccccccc}
               \lambda_1 \alpha & \ldots & \lambda_{s_1} \alpha        & 0           & \ldots & 0               & \nu_1 a & \ldots & \nu_{s_3} a   & 0 & \ldots & 0             \\
                    0           & \ldots &    0                        & \mu_1 \beta & \ldots & \mu_{s_2} \beta & \nu_1 b & \ldots & \nu_{s_3} b   & 0 & \ldots & 0
             \end{array} \right),
 \notag
    \end{equation}
    for suitable $\lambda_i, \mu_j, \nu_k, \alpha, \beta, a, b\in \K^*$   and $ s_1,s_2,s_3 \in \Z_{\ge 0}$. Then
  \begin{equation}
          X = \diag( a, b    ) \cdot \left(\begin{array}{cccccccccccc}
                    1  & \ldots & 1 &    0  & \ldots & 0 & 1 & \ldots & 1 & 0 & \ldots & 0  \\
                    0  & \ldots & 0 &    1  & \ldots & 1 & 1 & \ldots & 1 & 0 & \ldots & 0 
             \end{array}\right) \cdot D,
  \end{equation}
 where $D = \diag\left(  \frac{\lambda_1\alpha} a, \ldots,  \frac{\lambda_{s_1}\alpha}{a}, 
     \frac{\mu_1\beta}{ b}, \ldots,  \frac{\mu_{s_2}\beta}{b}, \nu_1, \ldots, \nu_{s_3},1,\ldots, 1 \right)$.
    This is the shape we were looking for.
\end{proof}
  We will also need the following variant of Corollary~\ref{Corollary:BLrepresentationOfRegularMatroid}.
  \begin{Corollary}
 \label{Corollary:BLrepresentationOfRegularMatroidTUM}
  Let $ X, A \in \K^{d\times N}$ be two matrices that represent the same labelled regular matroid $M$ of rank $d$.
  Representing the same labelled matroid means that 
  for every $I\subseteq \binom{[N]}{d}$, 
  the submatrix $X|_{I}$ is a basis if and only if the submatrix $A|_{I}$ is a basis.
  Furthermore, we assume that $A$ is totally unimodular.

 Then there exist 
  a non-singular diagonal matrix $D\in \GL(N,\K)$,
  and a matrix $T\in \GL(d,\K)$ 
  \st
  $ X =  T A D$.
\end{Corollary}
\begin{proof}
  If $d\ge 3$, this follows directly from   \cite[Proposition~6.3.13]{MatroidTheory-Oxley}.
  If $d\le 2$, $D$ and $T$ can be constructed explicitly as in the proof of Corollary~\ref{Corollary:BLrepresentationOfRegularMatroid}.
  The assumption that $X_2$ is totally unimodular
  allows us to drop the field automorphism (totally unimodular matrices are invariant under field automorphisms). 
\end{proof}

 \begin{Remark}
 \label{Remark:RegularDecomposition}
   The proof of    Corollary~\ref{Corollary:BLrepresentationOfRegularMatroid}
   can be made constructive.
   First note that $X = T A D$ if and only if 
   $X = (\lambda T) A (\frac 1 \lambda D)$ for any $\lambda\in \K^*$, where $\lambda$ acts on the matrices by scalar multiplication.
   Hence we cannot expect $T$ and $D$ to be uniquely determined.
   Loops in the matroid correspond to zero columns in $X$ and $A$, so they are not affected by $T$ and $D$. Therefore, we may assume that there are no loops.

   The matrix $A$ can be found as follows: pick a basis  $ B_0$ of the matroid $M$. \WLOG, its elements correspond to the first $d$ columns of $X$.
   We set the corresponding columns of $A$ to an identity matrix.
   Now consider an entry $a_{ij}$ of $A$ with $j>d$. If
   $ B_0 \setminus \{i\} \cup \{j\}$ is dependent, $a_{ij}=0$ must hold, otherwise $a_{ij}=\pm 1$.
   This leads to a (potentially pretty large) number of candidates for the matrix $A$. Since the matroid is regular, one of these matrices
   must be a totally unimodular matrix that represents the matroid.
   In fact, one can fix the signs of certain entries arbitrarily (on a so-called coordinatizing path) and
   then the remaining signs are uniquely determined
   \cite[Proposition~2.7.3]{brylawski-lucas-1976} (see also \cite[Theorem~6.4.7]{MatroidTheory-Oxley}).

   We will now explain how one can obtain the matrices $D=\diag(\delta_1,\ldots, \delta_N)$ and $T\in \GL(d,\K)$.   
   Let $B$ and $B' = B \setminus\{i\} \cup \{j\}$ denote two distinct bases of the matroid represented by $X$.
   Let $A|_B$ and $A|_{B'}$ denote the submatrices of $A$ whose columns are indexed by $B$ and $B'$, respectively. 
   Similarly, we define $X|_B$ and $X|_{B'}$.
   Let $D_{B}$ and $D_{B'}$ denote the square submatrices of $D$  whose rows and columns are indexed by $B$ and $B'$, respectively.
   Then we have 
   \begin{align}
      \det(X|_{B}) &= \det(T)\det(A|_{B})\det(D|_B) =  \det(T) \det(A|_{B}) \prod_{\nu\in B} \delta_\nu 
      \\ &
    =  \det(T) \det(A|_{B}) \frac{ \delta_i   }{ \delta_j   }  \prod_{\nu\in B'} \delta_\nu =    \frac{\det(A|_{B})}{\det(A|_{B'})} \frac{ \delta_i   }{ \delta_j   } \det(X_{B'}) .
   \end{align}
    Hence $ \frac{ \delta_i   }{ \delta_j   } = \frac{\det(X|_B)\det(A|_{B'})}{\det(X|_{B'})\det(A|_{B})}$.   
   Since the basis exchange graph of a matroid is connected \cite{maurer-1973},
   we can find all the $\delta_i$ after setting $\delta_1:=1$. 
   Once we know $D$, we can fix a basis $B$ and using the equation
   $X|_B = T \cdot A|_B \cdot  D|_B$, we obtain $T =   X|_B (A|_B D|_B)^{-1}$.  
 \end{Remark}

 \subsection{Proof of Theorem~\ref{Proposition:MatricesTwoMultiplicities}}
 Using Corollary~\ref{Corollary:BLrepresentationOfRegularMatroidTUM}, it is relatively easy to prove
 this result.
\begin{proof}[Proof of Theorem~\ref{Proposition:MatricesTwoMultiplicities}]
  By assumption, $X_1$ and $X_2$ define the same regular labelled matroid that can be represented by a totally unimodular matrix $A$.
  By Corollary~\ref{Corollary:BLrepresentationOfRegularMatroidTUM},
  there are matrices $T_1,T_2\in \GL(d,\K)$ and
  diagonal matrices $D_1=\diag(\delta_1,\ldots, \delta_N),D_2=\diag(\delta_1',\ldots, \delta_N')\in \GL(N,\K)$
  \st
  $ X_1 = %
  T_1 A D_1$
  and
  $ X_2 =  T_2 A D_2$.
  Now let 
  \begin{align}
  X_{12} &:= T_1^{k_1} T_2^{k_2} A D_1^{k_1} D_2^{k_2}.  \notag
  \text{ Then for a set $I\in \binom{[N]}{d}$, 
 } 
  \displaybreak[2]\\
 \Delta_I(X_1) &=  \det(T_1) \det(A|_I) \prod_{i \in I} \delta_i,\\ %
 \Delta_I(X_2) &=  \det(T_2) \det(A|_I) \prod_{i \in I} \delta_i', \text{ and}   \displaybreak[2]\\
 \Delta_I(X_{12}) &=  \det(T_1)^{k_1} \det(T_2)^{k_2} \det(A|_I) \prod_{i \in I} \delta_i^{k_1} \delta_i'^{k_2} \text{ holds.}
  \end{align}
 This implies the first two statements.
 For the second statement, we use that $\det(A|_I)^{k_1+k_2} = \det(A|_I)$ if $k_1+k_2$ is odd.
 \smallskip
 
 For an ordered field $\K$, let $\sgn : \K \to \K$ denote as usual the function that maps $0$ to $0$ and $x\in \K^*$ to $x/\abs{x}$.
   For $j\in \{1,2\}$, let $\tilde T_j$ denote an identity matrix whose top-left entry is 
  replaced by $\sgn(\det(T_j)) \cdot \abs{\det(T_j)}^k$. Of course, $(T_j)^k$ and $\tilde T_j$ have the same determinant (up to sign) if they exist, 
  but in the setting of (iii), $(T_j)^k$ may not be defined.
  Let $\tilde D_1 := \diag(\sgn(\delta_1)\cdot\abs{\delta_1}^{k_1},\ldots, \sgn(\delta_N)\cdot\abs{\delta_N}^{k_2})$ and
   $\tilde D_2 := \diag(\sgn(\delta_N')\cdot\abs{\delta_1'}^{k_2},\ldots, \sgn(\delta_N')\cdot\abs{\delta_N'}^{k_2})$.
  Now let $X_{k_1k_2} := \tilde T_1 \tilde T_2  A  \tilde{D_1} \tilde {D_2}$.
  By construction, 
 \begin{equation}
      \abs{\Delta_I(X_{12})} =  \abs{\det(T_1)}^{k_1} \abs{\det(T_2)}^{k_2} \prod_{i \in I} \abs{\delta_i}^{k_1} \abs{\delta_i'}^{k_2} 
 \end{equation}
 for all $I\in \binom{[N]}{d}$
 and the signs are preserved.
\end{proof}

\section{On representable arithmetic matroids}
 \label{Section:MultiplicityFunctions}

 In this subsection  we will explain how the multiplicity function  can be calculated from a representation
  of an arithmetic matroid as greatest common divisor (gcd) of certain subdeterminants.
  Additionally, we show that a representable arithmetic matroid that has a multiplicative basis 
  can be represented by a matrix that starts with a diagonal matrix.
 
 Recall that the \emph{greatest common divisor} of two or more integers, not all of them zero, 
 is the largest \emph{positive} integer that divides all of them.
 We will use the following statement that is very easy to prove. 
\begin{Lemma}
  \label{Lemma:gcdPower}
  Let $I$ be a finite set and let $A = ( a_i)_{ i \in I}$ be a list of integers.
  Let $k \ge 0$ be an integer. 
  Then 
  $\gcd( a_i^{k} : i \in I ) =  \gcd( a_i : i \in I)^{k} $.
\end{Lemma}

 Recall that for a list of vectors $X\subseteq \Z^d$ or a list of elements of a finitely generated abelian group $X\subseteq G$, 
 $m_X$ denotes the multiplicity function of the arithmetic matroid represented by $X$.
\begin{Lemma}
 \label{Lemma:MultiplicityFreeGroup}
 Let $X \subseteq \Z^d$ be a list of vectors and let $A\subseteq X$.
 \begin{enumerate}[(i)]
  \item If $A$ is independent, then  $m_X(A)$ is the greatest common divisor of all minors of size $\abs A$ of the matrix $A$.
         \label{Lemma:MultiplicityFreeGroupIndependent}
  \item  For arbitrary $A\subseteq X$, we have
        \begin{equation}
 \label{eq:dependentmultiplicity}
 m_X(A) = \gcd( \{ m_X(B) : B \subseteq A \text{ and  } \abs{B} = \rank(B) = \rank(A) \}).
 \end{equation}
\end{enumerate}
\end{Lemma}
  The first part is essentially due to Stanley \cite[Theorem~2.2]{stanley-1991}, the second was observed
  by D'Adderio--Moci \cite[p.~344]{moci-adderio-2013}.
  This lemma allows us to calculate the multiplicity function of an arithmetic matroid, given a representation in a free group.
  Recall that a representable arithmetic matroid can be represented in a free group if and only if $ m( \emptyset ) = 1 $, since
  $ m( \emptyset ) $ is equal to the cardinality of the torsion subgroup of the ambient group.

 If there is torsion, the calculation becomes a bit more complicated.  
  The first of the following two lemmas is an important special case of the second. 
  We will still give a proof for the first lemma, since it is very short.
\begin{Lemma} 
 \label{Lemma:TorsionIndependent}
  Let $G$ be a finitely generated abelian group and $X$ be a finite list of elements of $G$.
  Recall that $G_t$ denotes the torsion subgroup of $G$ and $\latproj{X}$ denotes the image of $X$ in $G/G_t$ (cf.~Subsection~\ref{Subsection:AriMatroids}).

  Then
  for $A\subseteq X$ independent, 
  $m_X(A) = m_{\latproj{X}}(A) \cdot  \abs{G_t}$.
\end{Lemma}
\begin{proof}
  If $A$ is independent, then $G_A \cong \bar G_{\latproj A}  \oplus G_t$ and  
   $ G \supseteq \langle A\rangle \cong \langle \latproj A\rangle \subseteq{\latproj{G}}$.
  Hence   $m_X(A) = \abs{ G_A / \langle A\rangle  } =  \abs{\bar G_{\bar A} / \langle \bar A\rangle} \cdot \abs{G_t} =
    m_{\latproj{X}}(A) \cdot  \abs{G_t}$.
\end{proof}

 Let $X  \subseteq \Z^d \oplus \Z_{q_1}\oplus \ldots \oplus \Z_{q_n}$ be a list  with $N$ elements. 
 On page~\pageref{page:FirstDefinitionOfLifting}, we defined a lifting of $X$. 
 Now we will describe a concrete construction of the lifting:
 we can choose $\lift(X) \in\Z^{(d+n)\times (N+n)}$ as the matrix 
 \begin{equation}
 \label{equation:LiftingConstruction}
    \lift(X) := \left(\begin{array}{c|c}  \latproj{X}  & 0 \\\hline 
                                                  L   & Q
         \end{array}\right) \in\Z^{(d + n)\times (N+n)},
 \end{equation}
 where $Q=\diag(q_1,\ldots, q_n)\in \Z^{n\times n}$
 and $L = (l_{ij})$ denotes a lifting of the torsion part of $X$, \ie  $l_{ij}$ is equal modulo $q_j$ to
 the $i$th torsion part of the $j$th vector in $X$.
 The arithmetic matroid $\Acal(X)$ is represented by the list of vectors $ \lift(X) / (0,Q)^T$, so $\lift(X)$ is indeed a lifting of $X$.
 The following lemma holds for any lifting of the list $X$.
 
 \begin{Lemma}[Lifting]
 \label{Lemma:generalmultiplicity}
   Let  $X \subseteq \Z^d \oplus \Z_{q_1}\oplus \ldots \oplus \Z_{q_n}$ and let  $\lift(X)\subseteq \Z^d\oplus \Z^n$ be a lifting.
  Let   $\Acal_X= (E, \rank, m_X)$ and  $\Acal_{\lift(X)}= (E \cup\{\iota_1,\ldots, \iota_n\}, \rank_{\lift(X)}, m_{\lift(X)})$ denote the corresponding arithmetic matroids.
   Let  $A \subseteq E$. Then 

   \begin{equation*}
   \begin{split}
            m_X(A) &= m_{\lift(X)}(A \cup \{\iota_1,\ldots, \iota_n\} )  = \gcd( \{  \det(S) : S \text{ maximal square sub-} \\
                  & \qquad\quad  \text{matrix of maximal independent subset of } \lift(X)|_{ A \cup \{\iota_1,\ldots, \iota_n\}} \}).      
   \end{split}
   \end{equation*}
 \end{Lemma}
 \begin{proof}
   The first equality follows from the fact that 
   \begin{equation*}
    \Acal(X) = \Acal({\lift(X)}) /\{\iota_1,\ldots, \iota_n\}   
   \end{equation*}
 holds and the definition of the contraction.
   The second equality follows from 
     Lemma~\ref{Lemma:MultiplicityFreeGroup}.
 \end{proof}

 \begin{Lemma}[QSUL]
  \label{Lemma:MultiplicityScaledUnimodular}
  Let $X = (x_e)_{e\in E} \subseteq \Z^d \oplus \Z_{q_1}\oplus \ldots \oplus \Z_{q_n}$ be a QSUL with $N$ elements, \ie
  there is a totally 
  unimodular list $A \in \Z^{(d+n)\times (N+n)}$, a non-singular diagonal matrix $D=\diag(\delta_1,\ldots, \delta_{N+n})\in \Z^{(N+n)\times (N+n)}$ and a sublist $Y\subseteq AD$ 
  \st
   $X = A D / Y$.
     Then the multiplicity function $m_X$ satisfies
     \begin{equation}
      m_X(S) = \gcd \bigg( \bigg\{ \prod_{ e \in T } \delta_e : T \text{ maximal independent subset of } S \cup Y \bigg\} \bigg)      
     \end{equation}
    for any $ S\subseteq E$.
  \end{Lemma}
\begin{proof}
     By Lemma~\ref{Lemma:generalmultiplicity}
     \begin{equation}
     \begin{split}  m_X(S) &= 
               \gcd(  \{ \det(K) : K \text{ maximal square submatrix of maximal } \\ 
               & \qquad\qquad \text{independent subset of } (AD)|_{ S \cup Y } \} ).     
             \end{split}
     \end{equation}
     Let $T \subseteq  (AD)|_{ S \cup Y } $ be a maximal independent subset.
     It is sufficient to prove that 
     \begin{equation}
             \prod_{e\in T} \abs{ \delta_e } = \gcd( \{ {\det(K)} : K \text{ maximal square submatrix of } T \}). 
     \end{equation}
     But this is clear: we can write $K = A_1 D_1$, with $A_1$ and $D_1$ suitable submatrices of $A$ and $D$.
     By construction $ \det(K) = \det(A_1) \prod_{j \in J} \delta_j $, where $J$ indexes the elements of $T$.
     Since $A$ is totally unimodular, $\det(A_1) \in \{0,-1,1\}$ must hold. Since $T$ is independent, 
     there must be one submatrix $K$ \st the corresponding  $A_1$ has a 
     non-zero determinant. This finishes the proof.
  \end{proof}

\begin{Lemma}
   \label{Lemma:DiagonalMatrixMultiplicativeBasis}
   Let $X \subseteq \Z^d$ be a list of vectors that spans $\R^d$ and let $B$ be a multiplicative basis for the arithmetic matroid  $\Acal(X) = (E, \rank, m)$. 
   Let $X'$ denote the Hermite normal form of $X$ with respect to $B$.
   Then 
   the columns of $X'$ that correspond to $B$ form a diagonal matrix.
 \end{Lemma}
 \begin{proof}
  Let $(b_1,\ldots, b_d)$ denote the columns of $X'$ that correspond to $B$, \ie the ones that form an upper triangular matrix.
   Let $\lambda_1,\ldots, \lambda_d$  denote the entries on the diagonal of this matrix ($\lambda_i\in \Z_{\ge 1}$).
   Then
   \begin{equation}
     \label{eq:HNFmultiplicity}
     \prod_{i=1}^d \lambda_i = \det(B) = m(B) = \prod_{i=1}^d m(\{ b_i\}). 
   \end{equation}
   Recall that $m(\{ b_i\})$ is equal to the absolute value of the $\gcd$ of all its entries.
   Since all entries of $b_i$ are non-negative and $\lambda_i$ is strictly bigger then all the other entries, we have
   $\lambda_i \ge  m(\{ b_i\})$ and equality holds if and only if $\lambda_i$ is the only non-zero entry.
   Hence \eqref{eq:HNFmultiplicity} can hold only if $\lambda_i =  m(\{ b_i\})$ for all $i$. This implies that the columns of $X'$ that correspond to $B$ form a diagonal matrix.
 \end{proof}

 \section{On multiplicative and regular arithmetic matroids}
 \label{Section:MultiplicativeRegularAriMatroids}
 
 In the first subsection we will prove the results on general regular arithmetic matroids
 that are strongly or weakly multiplicative.
 In Subsection~\ref{Subsection:LabelledGraphs} we will prove the results on  
 arithmetic matroids defined by labelled graphs.
 
 \subsection{On general multiplicative and regular arithmetic matroids}
  
\begin{proof}[Proof of Proposition~\ref{Proposition:RegularDoublyScaledUnimodular}]

   (i)$\Rightarrow$(ii):  
   Let us first assume that the arithmetic matroid $\Acal$ is torsion-free.
   Then it can be represented by a list of vectors $X\subseteq \Z^d$. 
   Without loss of generality, the first $d$ columns of $X$ form a multiplicative basis.
   Hence by Lemma~\ref{Lemma:DiagonalMatrixMultiplicativeBasis}, we may assume that
   they
   form a diagonal matrix $X_1\in \Z^{d \times d}$.

  The matrix $X$ is a rational representation of a regular matroid.
  Hence by Corollary~\ref{Corollary:BLrepresentationOfRegularMatroid},
  there   is a totally unimodular matrix
  $A\in \Z^{d\times N}$,
    $T\in \GL(d,\Q)$, and a diagonal matrix $D\in \GL(N,\Q)$ 
   \st $X = T A D$. 
  We may assume that the first $d$ columns of $A$ form an identity matrix:
   by assumption, they form a basis $B$ which has determinant $\pm 1$.
   If this is not the identity matrix, we can replace $A$ by $ B^{-1}A $ and $ T $ by $ TB $.
   $B^{-1}A$ is still totally unimodular and of course, $ X = T B B^{-1}A D$ holds.
   Let $D_1$ denote the $(d\times d)$-submatrix of $D$ that consists of the first $d$ rows and columns.
   We have $X_1 = T D_1$. Since $X_1$ and $D_1$ are both non-singular diagonal matrices, $T$ must be a diagonal matrix too.
   Hence $X$ is a doubly scaled unimodular list.

  Now let us consider the case where $\Acal$ may have torsion.
  By definition, there is a torsion-free arithmetic matroid $\Acal'$ that has a multiplicative basis and there is a subset 
  $Y$ of this basis \st $\Acal = \Acal'/Y$.
  As we have shown above, $\Acal$ can be represented by a doubly scaled unimodular list, \ie 
  $X' = TAD$, where $A$ is totally unimodular and $T$ and $D$ are diagonal matrices. 
  In addition, we may assume that the columns that correspond to the multiplicative basis form a diagonal matrix.
  $Y$ is a subset of these columns.
  Hence $X = TAD/Y$ is a QDSUL.

  \smallskip
  (ii)$\Rightarrow$(i): Let $X = TAD/Y$ be a representation of $\Acal$ by a QDSUL. 
  Then $TAD$ represents a torsion-free lifting of $\Acal$.
  Using the identity matrix requirement of a QDSUL (cf.~\eqref{eq:ShapeQDSUL}) and Lemma~\ref{Lemma:generalmultiplicity},
  it is easy to see that the lifting has a multiplicative basis.
\end{proof}
\begin{proof}[Proof of Proposition~\ref{Proposition:RegularScaledUnimodular}]
  (ii)$\Rightarrow$(i)  
  Let $\Acal$ be the arithmetic matroid that is represented by the list $X=AD/Y$
  with $A$ totally unimodular, $D$ a non-singular diagonal matrix and $Y\subseteq AD$.
  It follows from Lemma~\ref{Lemma:MultiplicityScaledUnimodular} that $AD$ represents an arithmetic matroid that is strongly multiplicative.
  Hence $\Acal$ is a quotient of a torsion-free arithmetic matroid that is strongly multiplicative.

   \smallskip
  (i)$\Rightarrow$(ii).
    Let us first assume that the arithmetic matroid $\Acal$ is torsion-free.
    As in the proof of Proposition~\ref{Proposition:RegularDoublyScaledUnimodular},
    using Corollary~\ref{Corollary:BLrepresentationOfRegularMatroid},
    we  can  find  diagonal  matrices  $ T = \diag(t_1,\ldots, t_d) \in \GL(d,\Q)$ and $ D = \diag(\delta_1,\ldots, \delta_N) \in \GL(N,\Q) $  and 
    a totally unimodular matrix $A$ 
    \st $X = T A D$ and the first $d$ columns of $A$ form an identity matrix.
    It is sufficient to show that we can assume that $T$ is the identity matrix.
    We may assume that all $t_i$ are positive:
    if $t_i<0$, we can  replace it by $-t_i$ and $A$ by the matrix obtained from $A$ by multiplying the $i$th row by $-1$. This does not change the product $TA$.
    Now suppose there is an entry  $t_i \in \Q \setminus \{0,1\}$.
    
    \emph{Case 1:} there is $ j \in \{ d+1, \ldots, N \}$ and $ k\in [d] $ \st $ a_{ij}\neq 0 $, $ a_{kj} \neq 0 $, $\abs{t_i}\neq \abs{t_k}$. This means
    that in the $j$th column of $X$, the $i$th and the $k$th entry are non-zero and have different absolute values.
     Without loss of generality, $\abs{t_i} > \abs{t_k}$ (otherwise, we   switch $i$ and $k$).
     Now let us consider the basis $B = \{1,\ldots, \hat{i}, \ldots, d, j\}$. Since $a_{ij}\neq 0$, this is indeed a basis.
     Its multiplicity is
     \begin{align}
      \abs{ t_1\delta_1\cdots \widehat{t_i\delta_i}\cdots t_d\delta_d \cdot \delta_j t_i } 
     &\neq
    \abs{ t_1\delta_1\cdots \widehat{t_i\delta_i}\cdots t_d\delta_d \cdot \delta_j } \underbrace{\gcd(t_\nu : a_{\nu j}\neq 0)}_{ \le \abs{t_k} < \abs{t_i} } 
    \notag \\ &= \prod_{b\in B} m(b).
    \end{align}
    Hence the basis $B$ is not multiplicative. This is a contradiction.
    
    \emph{Case 2:} for all $ j \in \{d+1, \ldots, N\} $ and $i\in [d]$ \st $x_{ij}\neq 0$, all other entries in column $j$ of $X$
    are either $0$ or their absolute value is $\abs{x_{ij}}$.
    This implies that all  $t_\nu$ with $\nu \in \Gamma_{i} := \{ \nu \in [d] : \text{there exists } j \text{ \st } x_{ij} x_{\nu j} \neq 0\}$  are equal
    to $t_i$.
    In this case, we can just replace  
    all the $t_\nu$ with $\nu\in \Gamma_i$ by $1$ and all the  $\delta_j$  
    with $j\in \{ \mu \in [N] : \text{there is } \nu\in \Gamma_i \text{ \st } x_{\nu \mu}\neq 0\}$
    by $t_i\delta_j$.
    For the new diagonal matrix $D'$,
      $X=AD'$   holds.
    Here is an example for this process:
    \begin{align}
      &  \diag(2,2,2,5)
        \cdot 
       \begin{pmatrix} 
           1 & 0 & 0 & 0 & 1 & 1 & 0  \\
           0 & 1 & 0 & 0 & 1 & 0 & 0  \\
           0 & 0 & 1 & 0 & 0 & 1 & 0  \\
           0 & 0 & 0 & 1 & 0 & 0 & 1  \\
       \end{pmatrix} \cdot \diag(1,1,1,1,1,1,10)
       \displaybreak[2]
       \\ =&
       \diag(1,1,1,1) \cdot
       \begin{pmatrix} 
           1 & 0 & 0 & 0 & 1 & 1 & 0  \\
           0 & 1 & 0 & 0 & 1 & 0 & 0  \\
           0 & 0 & 1 & 0 & 0 & 1 & 0  \\
           0 & 0 & 0 & 1 & 0 & 0 & 1  \\
       \end{pmatrix} \cdot \diag(2,2,2,5,2,2,50).
   \end{align}
      
    \smallskip
  Now let us consider the case where $\Acal$ may have torsion.
    Let $X$ be a representation of $\Acal$ in a finitely generated abelian group and let
  $\lift(X)$ be a lifting \st all of its bases are multiplicative (exists by assumption).
  As we have seen above, we can write  $\lift(X) = A D$.
  So by the definition of a lifting we have $X = \lift(X) / Y = A D / Y$ for a suitable sublist $Y\subseteq AD$.  
\end{proof}

  \begin{Lemma}
  \label{Lemma:MultiplicityRegularMultiplicative}
     Let $\Acal = (E,\rank, m)$ be a regular and strongly multiplicative arithmetic matroid.
     By definition, there is a torsion-free arithmetic matroid 
     $\Acal' = (E\cup Y, \tilde\rank, \tilde m)$ that is strongly multiplicative
     and $\Acal= \Acal'/Y$, where $E \cap Y = \emptyset$.

     Then for any $S\subseteq E$, the multiplicity function $m$ satisfies
     \begin{equation}
      m(S) = \gcd \bigg( \bigg\{ \prod_{e\in T} \tilde m(e) : T \text{ maximal independent subset of } S \cup Y \bigg\} \bigg)  .    
     \end{equation}
  \end{Lemma}
\begin{proof}
 This follows directly from Lemma~\ref{Lemma:MultiplicityScaledUnimodular}, using the fact that $\Acal$ can be represented by a quotient of a scaled unimodular list
 (Proposition~\ref{Proposition:RegularScaledUnimodular}).
\end{proof}

\begin{Lemma}
  \label{Lemma:MinorWeaklyMultiplicative}
  Let $X\in \Z^{d\times N}$ be a  
  doubly scaled unimodular list, \ie there is a totally unimodular matrix
  $A \in \Z^{ d \times N }$ and two diagonal matrices of full rank $T = \diag(t_1,\ldots, t_{d}) \in \Q^{d\times d}$ and 
  $D=\diag(\delta_1,\ldots, \delta_{N}) \in \Q^{N\times N}$
  \st $X = T A D $.
  Let $  I \subseteq [d] $  and  $ J \subseteq [ N ] $  be 
  two sets of the same cardinality.
  Let $X[I,J]$ denote the minor of $X$ whose rows are indexed by $I$ and whose columns are indexed by $J$.
  Then $X[I,J]$ is either equal to zero, or it satisfies
  \begin{equation}
     \abs{X[I,J]} = 
         \Biggl|     \prod_{i\in I} t_i \prod_{j\in J} \delta_j    \Biggr|   .
  \end{equation}

\end{Lemma}
\begin{proof}
  For a doubly scaled unimodular list $X= TAD$, the matrix entries $x_{ij}$ and $a_{ij}$ satisfy
  $x_{ij} = t_i a_{ij} \delta_j$.
   Since the determinant is multilinear, we obtain $X[I,J]=   A[I,J] \,\cdot\,  \prod_{i\in I} t_i \prod_{j\in J} \delta_j $. Hence 
   if $A[I,J]$ is non-zero, this implies
   $\abs{X[I,J]}= \abs{ \prod_{i\in I} t_i \prod_{j\in J} \delta_j} $.
\end{proof}

\begin{proof}[Proof of Theorem~\ref{Theorem:PowerOfArithmeticMatroid}]
   Let $d$ denote the rank of $\Acal$.
   By Proposition~\ref{Proposition:RegularDoublyScaledUnimodular},
   $\Acal$ can be represented by a QDSUL  $ X \subseteq \Z^d \oplus \Z_{q_1} \oplus \ldots \oplus \Z_{q_n}$ for some $n, q_1,\ldots, q_n\in \Z_{\ge 0}$,
   \ie we can write $X = T A D / Y$ for suitable matrices $T=\diag(t_1,\ldots, t_{d+n})$, $A\in \Z^{(d+n)\times(N+n)}$
   totally unimodular, $D=\diag(\delta_1,\ldots, \delta_{N+n})$ and $Y\subseteq TAD$.
   We may assume that $A$ starts with a  $(d+n)\times (d+n)$-identity matrix and $Y$ is the sublist that consists of the 
   columns $d+1,\ldots, d+n$, \ie 
   $ Y =  ( q_1 e_{d+1},\ldots, q_n e_{d+n})$, 
   with $q_i = t_{d+i}\delta_{d+i}$ for $i\in [d]$.  As usual, $e_i$ denotes the $i$th unit vector.
   Let $Y_{k} := ( q_1^k e_{d+1},\ldots, q_n^k e_{d+n})   $ 
   and $ X_{k} := T^{k}  A D^{k}  / Y_{k}  \subseteq \Z^d  \oplus \Z_{q_1^k} \oplus \ldots \oplus \Z_{q_n^k}$.
   Since both $T$ and $D$ are diagonal matrices, $T^k A D^k$ has integer entries.
   It follows from 
     Lemma~\ref{Lemma:gcdPower},  Lemma~\ref{Lemma:generalmultiplicity}, and Lemma~\ref{Lemma:MinorWeaklyMultiplicative} 
   that $X_{k}$ represents the  arithmetic matroid $\Acal^k$.
\end{proof}

 \subsection{On arithmetic matroids defined by labelled graphs}
   \label{Subsection:LabelledGraphs}
   
   In this subsection we will prove the results on arithmetic matroids defined by labelled graphs that
   we stated in Subsection~\ref{Subsection:LabelledGraphsResults}.

\begin{proof}[Proof of Proposition~\ref{Proposition:LabelledGraphRepresentableRegularStronglyMultiplicative}]
  By definition, an arithmetic matroid defined by a labelled graph can be represented 
  by a QSUL.
  This implies regularity.
  Using   Proposition~\ref{Proposition:RegularScaledUnimodular}, this also implies 
  strong multiplicativity.
\end{proof}
 
  \begin{Lemma}
  \label{Lemma:MultiplicityLabelledGraph}
     Let $(\Gcal,\ell)$ be a labelled graph. Let $R$ denote its set of regular edges and $W$ its set of dotted edges.
     Let $\Acal(\Gcal,\ell)$ be the arithmetic matroid defined by this labelled graph.
     Then its multiplicity function satisfies
     \begin{equation}
      m(A) = \gcd \bigg( \bigg\{ \prod_{e\in T} \ell(e) : T \text{ maximal independent subset of } A \cup W \bigg\} \bigg)      
     \end{equation}
    for any $A\subseteq R$.
  \end{Lemma}
  \begin{proof}
    Since  $m(\{e\})=\ell(e)$ for an element $e$ of the ground set of $\Acal(\Gcal,\ell)$, this follows directly 
    from   Proposition~\ref{Proposition:LabelledGraphRepresentableRegularStronglyMultiplicative}
    and Lemma~\ref{Lemma:MultiplicityRegularMultiplicative}.
  \end{proof}

\begin{proof}[Proof of Proposition~\ref{Proposition:LabelledGraphsSquareable}]
 This follows directly from Lemma~\ref{Lemma:gcdPower} and 
 Lemma~\ref{Lemma:MultiplicityLabelledGraph}.
\end{proof}

\section{Towards a structural theory of arithmetic matroids}
\label{Section:Towards}

 Structural matroid theory is concerned with describing large families of matroids
 through certain structural properties. 
 This includes characterising matroids that are representable over a certain field through
 excluded minors (cf.~Theorem~\ref{Theorem:RegularExcludedMinors}).
 This line of research was inspired by earlier results in graph theory such as Kuratowski's characterisation
 of planar graphs through forbidden minors.
 Rota's conjecture states that representability of a matroid over a fixed finite field can be characterised by a finite list of excluded minors.
  A proof of this long-standing conjecture has recently been announced \cite{geelen-gerards-whittle}.
  For infinite fields, the situation is very different:
 Mayhew--Newman--Whittle proved that \emph{the missing axiom of matroid theory is lost forever},
  \ie   it is impossible to
 characterise representability  over an infinite field using a certain natural logical language  \cite{mayhew-newman-whittle-2018}.
 
 It is therefore an obvious question to ask, if there is a suitable axiom system that characterizes representable arithmetic matroids.
 The situation is a bit different from matroids, as there is no choice of the field involved.
 One cannot hope to find a simple method to decide if an arbitrary arithmetic matroid $ \Acal = (E, \rank, m)$ is representable.
 As every matroid can be turned into an arithmetic matroid by equipping it with the trivial multiplicity function $m\equiv 1$,
 this problem contains the question if a given matroid is representable over the rationals, which is impossible by the result mentioned
 in the previous paragraph.
 The following question is more interesting. 
 \begin{Question}
   Let $(E,\rank)$ be a matroid that is representable over the rational numbers.
   Is it possible to characterize the functions $m : 2^E \to \Z_{\ge 1}$ \st
   $\Acal = (E, \rank, m)$ is a representable arithmetic matroid?
 \end{Question}

 We are not able to answer this question in this article, but we will give some necessary conditions that must
 be satisfied by the multiplicity function of a representable arithmetic matroid. 
 This is somewhat similar to
 Ingleton's and Kinser's inequalities  
 \cite{cameron-mayhew-2016,ingleton-1971,kinser-2011} that must be satisfied by the rank function of a representable
 matroid.

\begin{Lemma}
\label{Lemma:GPrealisabilityCondition}
  Let $\Acal= (E, \rank, m)$ be a representable arithmetic matroid
  and
  let $r\ge 2$ be an integer. 
  Suppose that the  matroid $\Mcal = (E,\rank)$ has a minor $U$ of rank $r$ on $2r$ elements, \ie there are  
  disjoint subsets
  $I, J \subseteq  E$ \st $ (\Mcal/J)|_I = U$, $\abs{I}=2r$, and  $\rank(I\cup J) - \rank(J) = r$.

  Then for any partition $ S \cup T = I $ \st $ \abs{S} = r-1 $ and $ \abs{T} = r+1 $,
  there is a sign vector $ \sigma \in \{1,-1\}^\Tcal $ \st
  \begin{equation}
    \sum_{t\in \Tcal} \sigma_t  \cdot  m( S \cup \{ t \} \cup J) m(  ( T \setminus \{ t \} ) \cup J ) = 0,
  \end{equation}
 where $\Tcal := \{ t \in T :  T \setminus \{t\} \text{ and } S \cup \{t\} \text{ are independent in } \Mcal / J \}$.
\end{Lemma}
 Note that if $U$ is the uniform matroid $U_{r, 2r}$, then $T = \Tcal$.
\begin{proof}[Proof of Lemma~\ref{Lemma:GPrealisabilityCondition}]
 Suppose $\Acal$ is represented by a list of vectors $X = (x_e)_{e\in E}$ that is contained in a 
 finitely generated abelian group $G$.
 Then the minor $U$  is  represented by the  list of vectors $X' := (X/J)|_I \subseteq G / \langle J \rangle$.
 As usual, $\latproj X'$ denotes the image of the projection of $X'$ to the free group $( G / \langle J \rangle)/G_t$,  where $G_t$ denotes the torsion subgroup of $G/\langle J \rangle$.
 $\latproj X'$ spans a free group $G_U \cong \Z^r$.
 Let $B \subseteq I $ be a basis of $\latproj X'$.
 Using Lemma~\ref{Lemma:TorsionIndependent},
 we  obtain
 \begin{equation}
 \label{equation:MultiplicityQuotient}
   m_X( B \cup J) = m_{X'}( B ) =   m_\latproj{X'}(B) \cdot \abs{G_t}
   = \abs{\det(B)} \cdot \abs{G_t}.
 \end{equation}
 The Grassmann--Pl\"ucker relations (Theorem~\ref{Theorem:GPrelations}) for $\latproj X'$ (as a list of vectors in $G_U \otimes \R \cong \R^r$) imply  that
 for a suitable sign vector $\sigma\in \{1,-1\}^\Tcal$
\begin{equation}
   \label{equation:GPcondition}
    \sum_{ t\in \Tcal } \sigma_t \cdot \det (  S \cup \{ t \} )   \det (  T \setminus \{ t \} ) = 0 \text{ holds.}
\end{equation}
 Here,  $S \cup \{ t \}$ and  $T \setminus \{ t \}$ denote the two square submatrices of $\latproj{X}'$ whose columns are indexed by the two sets.
 It is sufficient to sum over $\Tcal$, since for $t \in T \setminus \Tcal$, 
 $\det (  S \cup \{ t  \} )   \det (  T \setminus \{ t  \} ) = 0 $.
 Multiplying \eqref{equation:GPcondition}  by $\abs{G_t}^2$ and using equation \eqref{equation:MultiplicityQuotient}, we obtain 
\begin{equation*}
    \sum_{ t\in \Tcal } \sigma_t \cdot m_X (  S \cup \{ t \} \cup J )   m_X (  ( T \setminus \{ t \}) \cup J ) = 0. \qedhere
\end{equation*}
\end{proof}

  We will 
  now define a necessary condition for representability that is based on 
  Lemma~\ref{Lemma:GPrealisabilityCondition}.

\begin{Definition}
  Let $\Acal= (E, \rank, m)$ be an arithmetic matroid with underlying matroid $ \Mcal=(E,\rank) $.
  We say that $\Acal$ is  \emph{$r$-Grassmann--Pl\"ucker}, or $(\GPr)$ for short, 
  if the following condition is satisfied:

  Let $U$ be a minor of rank $r$ on $2r$ elements, \ie
  there are  
  disjoint subsets
  $ I, J \subseteq E$ \st $ (\Mcal /J)|_I  = U$,
  $\abs{I}=2r$, and  $\rank(I\cup J) - \rank(J) = r$.
  Then
  for any partition $ S \cup T = I  $  \st $ \abs{S} = r - 1 $ and $ \abs{T} = r+1 $,
  there is a sign vector $ \sigma \in \{1, -1\}^\Tcal$ \st
  \begin{equation}
    \sum_{ t \in \Tcal } \sigma_t  \cdot  m( S \cup \{ t \} \cup J ) m( T \setminus \{ t \} \cup J ) = 0,
  \end{equation}
 where $\Tcal := \{ t \in T :  T \setminus \{t\} \text{ and } S \cup \{t\} \text{ are independent in } \Mcal / J \}$.
\end{Definition}

 Recall that we proved
 Theorem~\ref{Theorem:NonRegularNonRepresentable} by showing that
 if $\Acal=(E,\rank,m)$ is representable, then for any non-negative integer $k\neq 1$, $\Acal^k=(E,\rank,m^k)$ does not satisfy $(\GP\!_2)$.

  The conditions in Lemma~\ref{Lemma:GPrealisabilityCondition}
  are not sufficient for representability. 
  There are arithmetic matroids of rank $1$ that are not representable.
  Hence they 
  trivially satisfy $\GPr$ for all $r \ge 2$.
  An example of such an arithmetic matroid is $\Acal_{12}$ in 
  Example~\ref{Example:ArithmeticMatroidsNoTwoFunctions}.

 Very recently, Pagaria introduced the class of orientable arithmetic matroids. They satisfy a property 
  that implies $(\GPr)$ for all $r$ \cite{pagaria-orientable-2018}.

 \begin{Remark}
  Grassmann--Pl\"ucker relations also play an important role in the theory of
  oriented matroids, valuated matroids, and more generally, matroids over hyperfields. 
  Let $E$ be a finite ground set, $r\ge 0$ an integer, and let $\K$ denote the Krasner hyperfield.
  Baker and Bowler showed that 
  there is a natural bijection between equivalence classes of 
  alternating non-zero functions $\varphi: E^r\to \K$
  that satisfy the Grassmann--Pl\"ucker relations 
  and matroids of rank $r$ on $E$
  \cite{baker-bowler-2019}. 
  As both, their work and ours, deal with Grassmann--Pl\"ucker relations and matroids,
  it would be interesting to find a connection.
  However, Baker and Bowler point out that their theory is quite different from 
  the theory of arithmetic matroids and their generalisation, matroids over a ring 
    \cite[Section~1.7]{baker-bowler-2019}.
 \end{Remark}

 \begin{Remark} 
  A matroid over $\Z$ can be seen as an arithmetic matroid that has some additional structure, \eg 
  a finitely-generated abelian group is attached to each subset of the ground set \cite{fink-moci-2016}.
  In certain cases,
  \eg when counting   generalized flows and colourings on a list of elements of a finitely generated abelian group,
  this additional structure is required to obtain interesting combinatorial information
  \cite[Remark~7.2]{branden-moci-2014}.
  It is natural to ask whether a matroid over $\Z$ satisfies $\GPr$ or vice versa, an arithmetic matroid that satisfies
  $\GPr$ can be equipped with a matroid over $\Z$ structure.
  In general, matroids over $\Z$ do not satisfy $\GPr$. An example is the arithmetic matroid
  with underlying matroid $U_{2,4}$, whose basis multiplicities are all equal to $1$, except for one, which is
  equal to some $t\ge 3$. This arithmetic matroid can be equipped with a uniquely determined matroid over $\Z$ structure, but it does not satisfy
  $(\GP\!_2)$.
 On the other hand, the arithmetic matroid in 
 Example~\ref{Example:ArithmeticMatroidsNoTwoFunctions}
 trivially satisfies $\GPr$ for all $r$, but it is not 
  a matroid over $\Z$: 
  if this was the case, there would be a $\Z$-module $M(A)$ attached to each $A\subseteq E$ and these modules would satisfy certain conditions.
  In particular, 
  $M(\{1\}) \cong \Z \oplus \Z/6\Z$,  $M(\{1,2\})=\{ 0 \}$ and there is $x\in M(\{1\})$ \st 
  $M(\{1\})/(x) \cong M(\{1,2\})$. This is not possible.
 \end{Remark}

\paragraph{Acknowledgements}

The author would like to thank Spencer Backman for raising the question whether
given a representable arithmetic matroid $\Acal$, the arithmetic matroid $\Acal^2$
is also representable.
He is grateful to 
 Alessio D'Al\`i,
 Emanuele Delucchi,
 Jan Draisma,
 Linard Hoessly,
 Michael Joswig,
 Lars Kastner,
 Martin Papke,
 Elia Saini,
 and
 Benjamin Schr\"oter
 for many interesting discussions and helpful advice.
 The mathematics software 
 {\tt sage} \cite{sage-73} was used to calculate some examples.
 \newcommand{\MR}[1]{} 

\section*{References}
\bibliographystyle{amsplain}
\providecommand{\bysame}{\leavevmode\hbox to3em{\hrulefill}\thinspace}
\providecommand{\MR}{\relax\ifhmode\unskip\space\fi MR }
\providecommand{\MRhref}[2]{%
  \href{http://www.ams.org/mathscinet-getitem?mr=#1}{#2}
}
\providecommand{\href}[2]{#2}

\end{document}